\documentclass[12pt,reqno]{amsart}
\advance \topmargin by -\headheight
\advance \topmargin by -\headsep
\evensidemargin \oddsidemargin
\usepackage{amsmath}
\usepackage{amsfonts}
\usepackage{stmaryrd}
\usepackage{amssymb}
\usepackage{amsthm}
\usepackage{csquotes}

\usepackage{mathrsfs}
\usepackage{dsfont}
\usepackage{bm}
\usepackage{cite}
\usepackage{soul}

\numberwithin{equation}{section}
\renewcommand\vec{\bm}
\newcommand{\n}[1]{\|{#1}\|}

\newtheorem{theorem}{Theorem}[section]

\newtheorem{lemma}[theorem]{Lemma}
\newtheorem{Proposition}[theorem]{Proposition}
\newtheorem{Conjecture}[theorem]{Conjecture}
\newtheorem{Corollary}[theorem]{Corollary}

\usepackage{mathtools}
\DeclarePairedDelimiter{\ceil}{\lceil}{\rceil}
\DeclarePairedDelimiter{\floor}{\lfloor}{\rfloor}

\title[Energy estimates in sum-product and convexity problems]{Energy estimates in sum-product and convexity problems}
\author[Akshat Mudgal]{Akshat Mudgal}
\date{} 
\address{Department of Mathematics, Purdue University, 150 N. University Street, West Lafayette, IN 47907-2067, USA }
\address{School of Mathematics, University of Bristol, Fry Building, Woodland Road, Bristol, BS8 1UG, UK}
\email{am16393@bristol.ac.uk, amudgal@purdue.edu}

\renewcommand\vec{\bm}

\begin{document}

\maketitle
\begin{abstract}
We prove a new class of low-energy decompositions which, amongst other consequences, imply that any finite set $A$ of integers may be written as $A = B \cup C$, where $B$ and $C$ are disjoint sets satisfying
\[ |\{ (b_1, \dots, b_{2s}) \in B^{2s} \ | \ b_1 + \dots + b_{s} = b_{s+1} + \dots + b_{2s}\}| \ll_{s} |B|^{2s - (\log \log s)^{1/2 - o(1)}} \]
and
\[ |\{ (c_1, \dots, c_{2s}) \in C^{2s} \ | \ c_1 \dots c_{s} = c_{s+1} \dots c_{2s} \}| \ll_{s} |C|^{2s - (\log \log s)^{1/2 - o(1)}}.\]
This generalises previous results of Bourgain--Chang on many-fold sumsets and product sets to the setting of many-fold energies, albeit with a weaker power saving, consequently confirming a speculation of Balog--Wooley. We further use our method to obtain new estimates for $s$-fold additive energies of $k$-convex sets, and these come arbitrarily close to the known lower bounds as $s$ becomes sufficiently large. 
\end{abstract}

\section{Introduction}

This paper investigates topics surrounding the sum--product phenomenon, a growing collection of results that concern themselves with notions of additivity and multiplicativity amongst algebraic sets. This was first studied by Erd\H{o}s and Szemer\'edi \cite{ES1983}, who analysed the cardinalities of the sumset $sA$ and the product set $A^{(s)}$, when $s$ is a natural number and $A$ is a finite, non-empty subset of integers. Here, we write
\[ sA = \{ a_1 + \dots + a_s \ | \ a_1, \dots, a_s \in A\} \ \text{and} \ A^{(s)}  = \{ a_1 \dots a_s \ | \ a_1, \dots, a_s \in A\}. \]
These sets measure the arithmetic structure of $A$, as evinced by the fact that $|sA| \ll_{s} |A|$ whenever $A$ is an arithmetic progression, and $|A^{(s)}| \ll_{s} |A|$ when $A$ is a geometric progression. Erd\H{o}s and Szemer\'edi conjectured that these two notions can not simultaneously occur for a given finite set $A$ of integers, whence, at least one of the sumset or the product set must be extremally large. 

\begin{Conjecture} \label{ersz}
Let $s \geq 2$, let $\epsilon >0$ and let $A$ be a finite subset of integers. Then 
\[ |s A| + |A^{(s)}| \gg_{s, \epsilon} |A|^{s - \epsilon}. \]
\end{Conjecture}

Since their influential work, this problem has been thoroughly studied and generalised, and in particular, the $s=2$ case of this conjecture has seen a lot of progress in recent times, with techniques arising from a variety of areas being utilised to tackle this question (see \cite{GS2016}, \cite{So2009}). The current best known result in the $s=2$ case is present in work of Rudnev and Stevens \cite{RS2020}, who showed that 
\[ |2A| + |A^{(2)}| \gg_{\epsilon} |A|^{4/3 + 2/1167 - \epsilon}, \]
whenever $A$ is a finite susbet of $\mathbb{R}$. The cases when $s \geq 3$ are in a more contrasting situation, with very few results studying this problem. One such result arose from the beautiful work of Bourgain--Chang \cite{BC2005}, which, in particular, implies that for any sufficiently large natural number $s$, and for any finite set $A$ of integers, we have
\begin{equation} \label{boc3}
 |s A| + |A^{(s)}| \gg_{s} |A|^{b_{s}}, 
 \end{equation}
where we may choose $b_{s} \gg (\log s)^{1/4}$. 
Since their result, there has only been one other improvement in this setting, namely, the work of P\'{a}lv\"{o}lgyi and Zhelezov \cite{PZ2020}, which allows one to take $b_{s} \gg (\log s)^{1 - o(1)}$. 
\par

We note that more robust notions of additivity and multiplicativity have been analysed in reference to these type of problems, and thus, given a natural number $s$ and finite set $A$ of real numbers, we define the $s$-fold additive energy $E_{s}(A)$ of $A$ to be
 \[ E_{s}(A) = |\{ (a_1, \dots, a_{2s}) \in A^{2s} \ | \ a_1 + \dots + a_{s} = a_{s+1} + \dots + a_{2s} \}|, \]
and the $s$-fold multiplicative energy $M_{s}(A)$ of $A$ to be
\[ M_{s}(A) = |\{ (a_1, \dots, a_{2s}) \in A^{2s} \ | \ a_1 \dots  a_{s} = a_{s+1} \dots a_{2s} \}|. \]
Noting a simple application of the Cauchy-Schwarz inequality, we see that
\begin{equation} \label{csref}
 E_{s}(A) |sA| \geq |A|^{2s}  \ \text{and} \ M_{s}(A) |A^{(s)}| \geq |A|^{2s} , 
 \end{equation}
and so, whenever the sumset or the product set is small, the respective energy must be large. Thus, these energies are concrete measures of additivity and multiplicativity. Similarly, whenever either of $E_{s}(A)$ or $M_{s}(A)$ is small, then $\max\{|sA|, |A^{(s)}|\}$ is large, and so, noting Conjecture $\ref{ersz}$, one might naively expect that given any finite set $A$ of $\mathbb{Z}$, either $E_{s}(A)$ or $M_{s}(A)$ must be small. However, we see that this is not the case, since we may consider the set $A_N = \{1, 2, \dots, N\} \cup \{N^2, \dots, N^N\}$, wherein, we have $E_{s}(A_N), M_{s}(A_N) \gg_{s} |A_N|^{2s-1}.$
 \par
 
Despite this obstruction, we are able to prove the next best alternative, that is, we are able to show that any finite set $A$ of integers may be partitioned into sets $B$ and $C$, such that $B$ has a small additive energy and $C$ has a small multiplicative energy.

\begin{theorem} \label{gemn}
Let $k \geq 1$ be a real number, let $q$ be some even natural number, and write $\Lambda = 6 + 25 \log q$, and $s =2^{5 + (1+120 \cdot 2^{\ceil{600qk\Lambda}}) (\ceil{\log k }+1)}.$  Then for every finite, non-empty set $A$ of integers, there exists pairwise disjoint sets $B,C$ such that $A = B \cup C$ and
\[ E_{q/2}(B) \ll_{k,q} |B|^{q-q/5} \ \text{and} \ M_{s}(C) \ll_{s} |C|^{2s - k}. \]
\end{theorem}

We can utilise this theorem to generalise the aforementioned result of Bourgain--Chang to the setting of energies, albeit with a weaker power saving.

\begin{Corollary} \label{bta}
Let $s$ be a sufficiently large natural number. Then for every finite set $A$ of integers, there exist disjoint sets $B$ and $C$ such that $A = B \cup C$ and
\[ E_{s}(B) \ll_{s} |B|^{2s - \eta_{s}} \ \text{and} \ M_{s}(C) \ll_{s} |C|^{2s - \eta_{s}}, \]
where $\eta_{s} \geq D(\log \log s)^{1/2}(\log \log \log s)^{-1/2}$, and $D>0$ is some absolute constant.
\end{Corollary}

We remark that while such decompositions have been extensively studied, Theorem $\ref{gemn}$ and Corollary $\ref{bta}$ seem to be the first results which allow the power saving $\eta_{s} \to \infty$ as $s \to \infty$, consequently confirming a speculation of Balog and Wooley \cite{BW2017} in the integer setting. In fact, it was the latter authors who originally studied these so called \emph{low energy decompositions}, and showed that given natural numbers $s_1, s_2 \geq 2$, every finite subset $A$ of $\mathbb{R}$ may be partitioned as $A = B \cup C$, where
\begin{equation} \label{balwo2}
 E_{s_1}(B) \ll |A|^{2s_1 - 1 - \delta} (\log |A|)^{1- \delta} \ \text{and} \ M_{s_2}(C) \ll |A|^{2s_2 - 1 - \delta}(\log |A|)^{1 - \delta}, 
 \end{equation}
with $\delta = 2/33$. Moroever, even though their work has since been further refined by multiple authors, in part due to its close connection to Conjecture $\ref{ersz}$, the best value of $\delta$ that is known to be permissible in $\eqref{balwo2}$ still equals $3/11$ (see \cite[Theorem $1.8$]{Xu2021}).
\par

We further note that these decompositions seem to be much harder to study as compared to the corresponding sumset-product set bounds. Firstly, the latter can be derived from the former using inequalities such as $\eqref{csref}$, and in fact, Corollary $\ref{bta}$ furnishes bounds of the shape $\eqref{boc3}$ in a straightforward manner, albeit with a weaker exponent of the form $b_{s} \gg_{s} (\log \log s)^{1/2 - o(1)}$. Secondly, unlike the sumset-product set case, there is no known analogue of Conjecture $\ref{ersz}$ in the energy setting, in part due to the fact that there exist arbitrarily large subsets $A$ of $\mathbb{N}$ such that for every decomposition of $A$ into disjoint sets $B$ and $C$, and for every choice of $s_1, s_2 \geq 2$, one has either
\[ E_{s_1}(B) \gg |A|^{s_1 + \frac{s_1 - 1}{3}} \ \text{or} \ M_{s_2}(C) \gg |A|^{s_2 + \frac{s_2- 1}{3}}. \]
We refer the reader to \cite{BW2017} for more details regarding this construction.
\par

Our methods can be further utilised to provide a qualitative equivalence between sumset-product set bounds of the shape $\eqref{boc3}$ and low energy decompositions as presented in our results above. In order to state this, we first record some notation, and so, given a set $\mathcal{Z}$ of real numbers, a natural number $m$ and a real number $b \geq 1$, we denote $(\mathcal{Z}, m,b)$ to be a \emph{good tuple} if for every non-empty, finite subset $A$ of $\mathcal{Z}$, we have
\begin{equation} \label{hypco}
 |m A | + |A^{(m)}| \gg_{m}  |A|^{b}.
 \end{equation}

\begin{theorem} \label{eric}
Let $(\mathcal{Z},m,b)$ be a good tuple, and let $k, s_1, s_2$ satisfy $k = b/30$ and $s_1 = 2^{5 + (1 + 500\ceil{k} s_2)(\ceil{\log k} + 1)}$ and $s_2 = 2^{5 + (1 + 120m)(\ceil{\log k} + 1)}.$ Then, for every finite, non-empty subset $A$ of $\mathcal{Z}$, there exist disjoint sets $B,C$ such that $A = B \cup C$ and
\[ E_{s_1}(B) \ll_{k,m} |B|^{2s_1 - k } \ \text{and} \ M_{s_2}(C) \ll_{k,m} |C|^{2s_2 - k+1}. \]
\end{theorem}

Thus, Theorem $\ref{eric}$, when combined with $\eqref{boc3}$, is able to deliver quantitatively weaker versions of Theorem $\ref{gemn}$ and Corollary $\ref{bta}$ in a straightforward manner. We end our discussion on this class of results by mentioning that we can use the ideas present in this paper to derive further arithmetic information concerning the sets $B$ and $C$ arising in the conclusion of Corollary $\ref{bta}$, see, for example, Theorem $\ref{zidt}$. We discuss this, along with other applications of our method in \S2. 
\par

We now move to another topic of interest in arithmetic combinatorics, that is, the study of additive properties of convex sets. Thus, given any finite subset $I$ of real numbers, and a function $f : \mathbb{R} \to \mathbb{R}$, we let $f(I) = \{ f(i) \ | \ i \in I\}$. With this definition in hand, we are interested in studying the sumsets $2I$ and $2f(I)$, when $f : \mathbb{R} \to \mathbb{R}$ is a strictly convex function. Writing $I_{N} = \{1, \dots, N\}$, this problem was first studied by Erd\H{o}s \cite{Er1977} in the case when $I= I_N$, who formalised the heuristic that convexity should perturb additive structure, at least to some degree, and presented the following conjecture. 

\begin{Conjecture} \label{ersum}
Let $\epsilon>0$ and let $f : \mathbb{R} \to \mathbb{R}$ be a strictly convex function. Then 
\[ |f(I_N) + f(I_N)| \gg_{\epsilon} N^{2 - \epsilon}. \]
\end{Conjecture}

Here, for all finite subsets $A,B$ of $\mathbb{R}$, we define 
$A*B = \{ a * b \ | \ a \in A, b \in A \}$, where the operation $* \in \{+,-, \cdot\}$. Furthermore, defining the set $\mathcal{N}_k = \{1, 2^k, \dots, N^k\}$ for every $k \in \mathbb{N}$, we may use the classical estimate $|\mathcal{N}_2 + \mathcal{N}_2| \ll N^{2 - o(1)}$ (see \cite{La1908}) to discern that Conjecture $\ref{ersum}$ is expected to be sharp. This has also been analysed in the more general case of $I$ being an arbitrary finite subset of $\mathbb{R}$, wherein, Elekes, Nathanson and Ruzsa \cite{ENR2000} used incidence geometric methods to prove that
\begin{equation} \label{elr}
 |f(I)+f(I)| \gg |I|^{3/2} K_2^{-1}, 
 \end{equation}
 with $K_2 = |2I|/|I|$. We see that choosing $f(x) = \log x$ gives sum-product type estimates, while setting $I=I_N$ delivers bounds for sumsets of convex sets. As in the case of Conjecture $\ref{ersz}$, much work has been done recently on refining such bounds, culminating in the work of Stevens and Warren \cite{SW2021}, who showed that
 \[ |f(I) + f(I)| \gg |I|^{30/19 -o(1)} K_2^{-1}, \]
 in process, recovering the best known result known towards Conjecture $\ref{ersum}$.
 \par

Furthermore, various authors have studied generalisations of estimates of the above kind, by imposing higher levels of convexity on the function $f$ (see, for instance, \cite{HRR2020}, \cite{BHR2021}, \cite{Ol2020}, \cite{Sh2020}). More precisely, given some interval $\mathcal{I} \subseteq \mathbb{R}$ and natural number $k \geq 2$, we write $f$ to be $(k-1)$-convex on $\mathcal{I}$ if all the derivatives $f^{(1)}, \dots, f^{(k)}$ exist and are non-vanishing on $\mathcal{I}$. Furthermore, given a non-empty, finite subset $I$ of $\mathcal{I}$, we write $K = |2I-I|/|I|$ and $A = f(I)$. 
\par

In this setting, Hanson, Roche-Newton and Rudnev \cite{HRR2020} proved that
\begin{equation} \label{hrnr}
 |2^{k-1} A - (2^{k-1}-1) A| \gg_{k}  |A|^{k}K^{-2^k + k +1} (\log |A|)^{-O_{k}(1)},
 \end{equation}
a bound that can be interpreted as a higher convexity version of $\eqref{elr}$. Moreover, this is sharp up to some multiplicative constant, since the function $f(x) = x^k$ can be seen to be $(k-1)$-convex, whence, we can choose $I = I_N$ and $A = \mathcal{N}_k$, and use the observation that $|m \mathcal{N}_k - n \mathcal{N}_k| \ll_{k,m,n} N^k,$ for each $m,n \in \mathbb{N}$. 
\par

Noting the above estimates, it is natural to ask whether an energy variant of such a bound can hold, and this is precisely the content of our next result.

\begin{theorem} \label{rtp}
Let $k \geq 2$ be a natural number, let $\mathcal{I} \subseteq \mathbb{R}$ be an interval and let $f: \mathcal{I} \to \mathbb{R}$ be $(k-1)$-convex on $\mathcal{I}$. Then, for every $s \in \mathbb{N}$ and every finite $I \subseteq \mathcal{I}$, writing $A = f(I)$ and $K = |2I-I|/|I|$, we have
\[ E_{s}(A) \ll_{s,k}  |A|^{2s - k + (4k-4) s^{-\eta_k}}+ K^{c_{s,k}}, \]
where $\eta_k \in (0,1)$ and $c_{s,k} >0$ are absolute constants.  
\end{theorem}

As before, we note that Theorem $\ref{rtp}$ is sharp up to a factor of $|A|^{(4k-4) s^{-\eta_k}}$, where the latter becomes arbitrarily small as $s$ becomes appropriately large. In order to see this, we can set $f(x) = x^k$ and $I= I_N$ in Theorem $\ref{rtp}$ to produce the bound
\[  E_{s}(\mathcal{N}_k) \ll_{s,k} N^{2s - k + (4k-4) s^{-\eta_k}},\]
whereupon, we may use the following elementary estimate to confirm our claim 
\begin{equation} \label{war2}
E_{s}(\mathcal{N}_k)  \geq N^{2s} |s \mathcal{N}_k|^{-1} \gg_{s,k} N^{2s-k}.
\end{equation}
\par

Furthermore, while we have not carefully optimised our arguments to obtain the best possible value of $\eta_k$ in Theorem $\ref{rtp}$, one could take $2^{-k} \ll \eta_k\ll 2^{-k}$. Thus, there exists $C>0$, such that for any $\kappa \geq 1$,  whenever $s \geq 2^{ C2^k \log (k\kappa)}$, we have $(4k-4)s^{-\eta_k} \leq 1/\kappa$, that is, 
\[ E_{s}(A) \ll_{s,k}  |A|^{2s - k + 1/\kappa} + K^{c_{s,k}}. \]
This can be compared with the recent work of Bradshaw, Hanson and Rudnev \cite[Theorem $1.4$]{BHR2021}, who, under the hypothesis of Theorem $\ref{rtp}$,  showed that
\begin{equation} \label{devt}
 E_{s}(A) \ll_{s,k} K^{2^k - 2k + 2\alpha_k} |A|^{2s - k + \alpha_k} , 
 \end{equation}
 where $s \geq 2^{k-1}$ and $\alpha_k = \sum_{j=1}^{k-1} j2^{-j}$. Since $1/2 \leq \alpha_k \leq 2$ for every $k \geq 2$, we see that $\eqref{devt}$ misses the lower bound in $\eqref{war2}$ by a factor of at least $|A|^{1/2}$, and so, our result provides sharper upper bounds when $s \geq 2^{C2^{k} \log k}$ and $K \ll |A|^{O_{s,k}(1)}$. On the other hand, we note that their result provides non-trivial bounds in a much larger regime than ours, that is, when $s \geq 2^{k-1}$, as well as that $\eqref{devt}$ exhibits a better dependence on the parameter $K$.
 \par
 
Our method also provides almost sharp bounds for additive energies  between multiple $(k-1)$-convex sets, see, for instance, Corollary $\ref{rtp2}$. Moreover, a key ingredient in the proofs of all of our aforementioned results is an inverse theorem that converts information on $s$-fold energies into bounds for many-fold sumsets. This is the content of Proposition $\ref{kp}$, which we present in \S2.
\par

We finish this section by providing a brief outline of our paper. We use \S2 to discuss some further applications of our method. In \S3, we record some preliminary lemmata that we will use frequently through the paper. We then employ \S4 to present the general set up for procuring our low energy decompositions. In \S5, we will use these ideas, in conjunction with the aforementioned work of Bourgain--Chang, to prove Theorem $\ref{gemn}$ and related results, such as Corollary $\ref{bta}$ and Theorem $\ref{zidt}$. Similarly, we utilise \S6 to proving Theorem $\ref{eric}$. We turn to the proof of Theorem $\ref{rtp}$ in \S7, and this will involve combining iterative applications of Proposition $\ref{kp}$ with sumset estimates of the form $\eqref{hrnr}$. Finally, we conclude this paper by recording the proof of Proposition $\ref{kp}$ in \S8.
 
\textbf{Notation.} In this paper, we use Vinogradov notation, that is, we write $X \gg_{z} Y$, or equivalently $Y \ll_{z} X$, to mean $|X| \geq C_{z} |Y|$ where $C$ is some positive constant depending on the parameter $z$. Moreover, for every $\theta \in \mathbb{R}$, we use $e(\theta)$ to denote $e^{2\pi i \theta}$, and for every non-empty, finite set $Z$, we use $|Z|$ to denote the cardinality of $Z$. 

\textbf{Acknowledgements.}  The author is grateful for support and hospitality from University of Bristol and Purdue University. The author would like to thank Trevor Wooley for his guidance and encouragement.



\section{Further applications and discussion}

We commence this section by presenting versions of results from \S1 that study mixed energies. One of the motivations behind this is to show that the sets $B,C$ arising from our decomposition in Corollary $\ref{bta}$ not only have low additive and multiplicative energies respectively, but also that this decomposition has a doubly orthogonal flavour, that is, $B$ and $C$ have a small mutual additive and multiplicative energy. We begin this endeavour by supplying some suitable notation, and thus, given any natural number $s$ and any finite subsets $A_1, \dots, A_{2s}$ of $\mathbb{R}$, we define
\[ E_{s}(A_1, \dots ,A_{2s}) = |\{ (a_1, \dots, a_{2s}) \in A_1 \times \dots \times A_{2s} \ | \ a_1 + \dots + a_{s} = a_{s+1} + \dots + a_{2s}\}|, \]
and
\[ M_{s}(A_1, \dots ,A_{2s}) = |\{ (a_1, \dots, a_{2s}) \in A_1 \times \dots \times A_{2s} \ | \ a_1  \dots  a_{s} = a_{s+1} \dots a_{2s}\}|. \]
Next, for all finite sets $B,C$ of real numbers, writing $B_{i} = B$ for every $1 \leq i \leq s$ and $C_i = C$ for every $1 \leq i \leq s$, we define
\[ E_{s}(B,C) = E_{s}(B_1, \dots, B_{s}, C_{1}, \dots, C_{s}) \ \text{and} \ M_{s}(B,C) = M_{s}(B_1, \dots, B_{s}, C_{1}, \dots, C_{s}). \]
We can interpret $E_{s}(B,C)$ and $M_{s}(B,C)$ as measures of additive and multiplicative interactions between the sets $B$ and $C$. 

\begin{theorem} \label{zidt}
Let $s \in \mathbb{N}$ be sufficiently large, and let $A \subseteq \mathbb{Z}$ be a finite set. Then there exist disjoint sets $B,C$ such that $A = B \cup C$, and 
\[ \max\{ E_{s}(B), M_{s}(C), E_{s}(B,C), M_{s}(B,C)\} \ll_{s} |A|^{2s - \gamma_{s}}, \]
where $\gamma_{s} \gg (\log \log s)^{1/2} (\log \log \log s)^{-1/2}$. Moreover, if $0 \notin A$, then for every $A_1, \dots, A_{2s} \in \{ B,C\}$, we either have	
\begin{equation} \label{blueoc}
 E_{s}(A_1, \dots, A_{2s}) \ll_{s} |A|^{2s - \gamma_{s}} \ \text{or} \ M_{s}(A_1, \dots, A_{2s}) \ll_{s} |A|^{2s - \gamma_{s}}. 
 \end{equation}
\end{theorem}

Thus, the sets $B$ and $C$ prescribed by Theorem $\ref{zidt}$ not only have small additive and multiplicative energies respectively, but there is also low arithmetic interaction between the two sets. Moreover, it is interesting to note that in the case of $\eqref{blueoc}$, we must restrict to the case when $0 \notin A$, but we will show that this condition is, in fact, necessary.
\par

Writing $A= \{0,1, 2, \dots, N\}$, we suppose that $A = B \cup C$ for some disjoint sets $B,C$, and without loss of generality, we may further assume that $|B| \geq N/2$. We now set $A_{i} = A_{i+s} = B$ for every $2 \leq i \leq s$, and if $0 \in B$, we set $A_1 = A_{s+1} = B$, otherwise we set $A_1 = A_{s+1} = C$. In either case, applying the Cauchy-Schwarz inequality gives us
\[ E_{s}(A_1, \dots, A_{2s}) \geq E_{s-1}(B) \geq  | B|^{2s-2} |sB|^{-1} \gg_{s} N^{2s-3}, \]
while
\[ M_{s}(A_1, \dots, A_{2s}) \geq |B|^{2s-2} \gg_{s} N^{2s-2}, \]
since $0 \cdot b_2 \dots b_{s} = 0 \cdot b_{s+1} \dots b_{2s}$ for every $b_2, \dots b_{s}, b_{s+2}, \dots, b_{2s} \in B$. This confirms our claim.
\par


As previously mentioned, we can also generalise Theorem $\ref{rtp}$ to the case when we consider additive interactions between multiple $(k-1)$-convex sets.  

\begin{Corollary} \label{rtp2}
Let $k,s \geq 2$ be natural numbers. Morevoer, for every $1 \leq i \leq 2s$, let $\mathcal{I}_i \subseteq \mathbb{R}$ be an interval, let $f_i: \mathcal{I}_i \to \mathbb{R}$ be $(k-1)$-convex, and let $I_i \subseteq \mathcal{I}_i$ be a finite set, with $A_i = f_i(I_i)$ and $K_i = |2I_i-I_i|/|I_i|$. Then, we have
\[ E_{s}(A_1, \dots, A_{2s}) \ll_{s,k} (K_1\dots K_{2s})^{O_{s,k}(1)} (|A_1|^{\frac{1}{2s}} \dots |A_{2s}|^{\frac{1}{2s}})^{2s - k + (4k-4) s^{-\eta_k}}, \]
where $\eta_k \in (0,1)$.
\end{Corollary}

\begin{proof}
Theorem $\ref{rtp}$ implies that $E_{s}(A_i) \ll_{s,k} K^{O_{s,k}(1)} |A_i|^{2s - k + (4k-4) s^{-\eta_k}}$ for every $1 \leq i \leq 2s$, which further combines with Lemma $\ref{yoc2}$ to deliver the required bound.
\end{proof}


Moreover, while Theorem $\ref{rtp}$ and Corollary $\ref{rtp2}$ and inequalities $\eqref{hrnr}$ and $\eqref{devt}$ can be seen as higher convexity generalisations of $\eqref{elr}$, the methods used herein differ significantly from the previous work done around these topics. For example, Elekes, Nathanson and Ruzsa used incidence geometric tools to prove $\eqref{elr}$, whereas the authors in \cite{HRR2020} and \cite{BHR2021} employed more elementary combinatorial methods in their work. On the other hand, we rely heavily on a variety of tools from arithmetic combinatorics in this paper, wherein, one of the key mechanisms uitilised is to convert information about many-fold additive energies into bounds for the corresponding sumsets. Such a philosophy is epitomised by a classical result in the area, known as the Balog--Szemer\'edi--Gowers theorem (see \cite{Bal2007}, \cite{Sch2015}), and we present the following generalisation of this result.
\par

\begin{Proposition} \label{kp}
Let $\nu, \delta$ be positive real numbers such that $\nu \geq 1$, and let $s \geq 4$ be some even number. Moreover, suppose that $A$ is some finite, non-empty set of real numbers such that $E_{s}(A) \geq |A|^{2s - \nu}$. Then we either have
\[ E_{s/2}(A) > |A|^{s-\nu  + \delta} \]
or there exists some $A' \subseteq A$ such that $|A'| \gg |A|^{1-82\delta},$ and for each $m, n \in \mathbb{N}\cup \{0\}$, we have
 \[ |m {A'} - n {A'}| \ll_{m,n}  |A|^{\nu + 240 (m+n) \delta}. \]
\end{Proposition}

 
Noting Lemma $\ref{yoc}$, we see that $E_{s}(A) \leq |A|^s E_{s/2}(A)$ for every even natural number $s$, and so, Proposition $\ref{kp}$ roughly states that whenever $E_{s/2}(A)$ and $E_{s}(A)$ follow an almost optimal relation, then we can obtain a very strong control over the many-fold sumsets of a large enough subset $A'$ of $A$. Another such inverse result arises in work of Shkredov \cite{Sh2013}, who showed that whenever $E_{2}(A)$ and the so-called higher energy follow an almost optimal relation, then one can obtain a conclusion of a similar strength as Proposition $\ref{kp}$. This result has found a variety of applications in arithmetic combinatorics and related areas, for instance, we point the reader to some of our recent work on threshold breaking estimates for additive energies of sets lying on curves and spheres (see \cite{Ak2020b}, \cite{Ak2021a}).  Similarly, Bateman and Katz \cite{BK2011} studied structure theorems for sets $\Lambda$ that, amongst other properties, satisfied almost optimal relations between $E_{2}(\Lambda)$ and $E_{4}(\Lambda)$, and this formed a crucial ingredient in their breakthrough work on the cap set problem. Relatives of these ideas also present themselves in the recent groundbreaking work of Bloom--Sisask \cite{BS2020} on Roth's theorem on arithmetic progressions.
\par

We end this section by noting that even in the case when $A = \mathcal{N}_k$, obtaining near optimal bounds for $E_{s}(\mathcal{N}_k)$ for smaller values of $s$, say, when $s=o(k^2)$, is incredibly hard, and any significant progress in that direction would imply major improvements in Waring's problem. Thus, writing $s_0(k)$ to be the smallest value of $s$ for which we have
\[ E_{s}(\mathcal{N}_k) \ll_{s,k, \epsilon} N^{2s - k + \epsilon}  \]
for every $\epsilon>0$, the problem of finding optimal estimates for $s_0(k)$ has been studied by multiple authors, with the current best known bounds arising from the work of Wooley \cite[Corollary $14.7$]{Wo2017}, who proved that $s_0(k) \leq k^2 - k +2\floor{\sqrt{2k+2}}$. We refer the reader to Section $14$ of the latter paper for more details regarding this topic.	
\par


\section{Preliminaries}

We begin by recording some preliminary definitions and results that we require in our proof of Theorem $\ref{rtp}$. Thus, given a finite, non-empty set $A$ of real numbers, for each $n \in \mathbb{R}$, we write $r_{s}(n)$ to denote the number of solutions to the equation $x_1 + \dots + x_s = n$ with $x_1, \dots, x_s \in A$. Thus, for any $n \in \mathbb{R}$, we have $r_{s}(n) \geq 1$ if and only if $n \in s {A}$. Moreover, a standard double counting argument implies that
\begin{equation} \label{sda}
 \sum_{n \in sA} r_{s}(n) = |A|^{s} \ \text{and} \ \sum_{n \in sA} r_{s}(n)^2 = E_{s}(A). 
 \end{equation}
Additionally, given a finite subset $X$ of $\mathbb{R}$, we use $\mathds{1}_{X}$ to denote the characteristic function of $X$, that is, given $n \in \mathbb{R}$, we have $\mathds{1}_{X}(n) = 1$ if $n \in X$, and $\mathds{1}_{X}(n) = 0$ otherwise. We now record some additive combinatorial inequalities concerning the relation between the representation function $r_{s}$ and the additive energy $E_{s}(A)$. 

\begin{lemma} \label{yoc}
Let $A$ be a set of real numbers and let $s$ be a natural number. Then, whenever $s$ is even, we have
\[ \sup_{n} r_{s}(n) \leq E_{s/2}(A). \]
Similarly, for each $1 \leq l < s$, we have
\[ E_{s}(A) \leq |A|^{2s-2l} E_{l}(A). \]
\end{lemma}
\begin{proof}
Let $G$ be the finitely generated abelian group spanned by elements of $A$. We begin by defining, for any finitely supported functions $f,g : \mathbb{R} \to \mathbb{R}$, the convolution
\[ (f*g) (x) = \sum_{n \in G} f(n) g(x-n) \ \text{for each} \ x \in \mathbb{R}. \]
Moreover, we can extend this definition for many-fold convolutions, by writing $f*_{1}f(x) = (f*f)(x)$ and $f*_{s}f (x) = (f * (f*_{s-1} f))(x)$ for each $s \geq 2$. Finally, for each $1 \leq p < \infty$, we denote the $l^p(G)$ norms of $f$ to be
\[ \n{f}_{p} = (\sum_{n \in G} |f(n)|^{p})^{1/p} \ \text{and} \ \n{f}_{\infty} = \sup_{n \in G} |f(n)|. \]
From these definitions, we see that $r_{s}(n) = (\mathds{1}_A *_{s-1} \mathds{1}_A) (n)$, and so, we can further rewrite $\eqref{sda}$ as
\[  \n{\mathds{1}_A *_{s-1} \mathds{1}_A}_{1} = |A|^s, \ \text{and} \  \n{\mathds{1}_A *_{s-1} \mathds{1}_A}_2^2 = E_{s}(A). \]
We can now use Young's convolution inequality to deduce that
\begin{align*}
\sup_{n} r_{s}(n)  = \n{\mathds{1}_A *_{s-1} \mathds{1}_A}_{\infty}   \leq \n{(\mathds{1}_A *_{s/2-1}\mathds{1}_A) }_{2}^2 = E_{s/2}(A), 
\end{align*}
whenever $s$ is even. Similarly, for every $1 \leq l < s$, we can again apply Young's convolution inequality to discern that
\begin{align*}
E_{s}(A)^{1/2} & = \n{\mathds{1}_A*_{s-1} \mathds{1}_A}_2  \leq \n{\mathds{1}_A*_{s-l-1} \mathds{1}_A}_{1} \n{\mathds{1}_A*_{l-1} \mathds{1}_A}_{2} = |A|^{s-l} E_{l}(A)^{1/2},
\end{align*}
which delivers the required estimate.
\end{proof}

We also prove upper bounds for $E_{s}(A_1, \dots, A_{2s})$ in terms of $E_{s}(A_1), \dots, E_{s}(A_{2s})$.

\begin{lemma} \label{yoc2}
Given finite, non-empty sets $A_1, \dots, A_{2s}$ of real numbers, we have
\[ E_{s}(A_1,\dots, A_{2s}) \leq E_{s}(A_1)^{1/2s} \dots E_{s}(A_{2s})^{1/2s}. \]
\end{lemma}
\begin{proof}
It suffices to show that for each $\epsilon>0$, we have
\[ E_{s}(A_1, \dots, A_{2s}) \leq  E_{s}(A_1)^{1/2s} \dots E_{s}(A_{2s})^{1/2s} + \epsilon. \]
We begin this endeavour by defining, for every pair $\xi, R$ of real numbers satisfying $\xi \neq 0$ and $R >0$, the quantity $I(R,\xi) = \int_{[0,R]} e(\xi \alpha) d \alpha.$ When $\xi \neq 0$, we see that
\begin{equation} \label{rnm4}
|I(R, \xi)| = | (2 \pi i \xi)^{-1}(e(\xi R ) - 1)| \ll  |\xi|^{-1}, 
\end{equation}
whereas $I(R, 0) = R$. Moreover, for every $1 \leq i \leq 2s$, we define the exponential sum $f_i(\alpha) = \sum_{a \in A_i} e( a\alpha)$, while we use $\mathcal{A}$ to denote the set $A_1 \times \dots \times A_{2s}$, and we use $\vec{a} = (a_1, \dots, a_{2s})$ to denote an element in $\mathcal{A}$. Next, we write
\[ \mathcal{U} = \{ \vec{a} \in \mathcal{A} \ | \ \sum_{i=1}^{s} (a_i - a_{i+s}) = 0\} , \ \text{and} \ \mathcal{V} = \{ \vec{a} \in \mathcal{A} \ | \ \sum_{i=1}^{s} (a_i - a_{i+s}) \neq 0\}, \]
and $T_0 = (A_1 + \dots + A_s - A_{s+1} - \dots - A_{2s})\setminus \{0\}$, and $T_i = (sA_i - sA_i) \setminus \{0\}$ for each $1 \leq i \leq 2s$.  Finally, we set $\xi_0 = \min_{\xi \in T_0} |\xi|,$ and $\xi_i = \min_{\xi \in T_i} |\xi|$, for each $1 \leq i \leq 2s$. 
\par

With this notation in hand, we see that
\begin{align} \label{hld33}
\int_{[0,R]} f_1(\alpha) \dots f_s( \alpha)  \overline{ f_{s+1}(\alpha) }\dots & \overline{f_{2s}(\alpha)} d \alpha  
 = \sum_{\vec{a} \in \mathcal{U}} I(R,0) + \sum_{\vec{a} \in \mathcal{V}} I(R, \sum_{i=1}^{s}(a_i - a_{i+s}))  \nonumber \\
& = R E_{s}(A_1, \dots, A_{2s}) + O(|A_1| \dots |A_{2s}| \xi_{0}^{-1}),
\end{align}
where the last step follows from $\eqref{rnm4}$. Similarly, for each $1 \leq i \leq 2s$, we have
\begin{equation} \label{hld4}
 \int_{[0,R]} |f_i(\alpha)|^{2s} d \alpha  = R E_{s}(A_i) + O(|A_i|^{2s}\xi_{i}^{-1}). 
 \end{equation} 
We may now use H\"{o}lder's inequality to deduce that
\[ \int_{[0,R]} f_1(\alpha) \dots f_s( \alpha) \overline{ f_{s+1}(\alpha) }\dots \overline{f_{2s}(\alpha)} d \alpha   \leq \prod_{i=1}^{2s} \Big( \int_{[0,R]} |f_i(\alpha)|^{2s} d \alpha \Big)^{1/2s}, \]
which then combines with $\eqref{hld33}$ and $\eqref{hld4}$ to deliver the estimate
\[ E_{s}(A_1, \dots, A_{2s}) + O(R^{-1}|A_1| \dots |A_{2s}| \xi_{0}^{-1}) \leq \prod_{i=1}^{2s} ( E_{s}(A_i) + O(R^{-1}|A_i|^{2s}\xi_{i}^{-1}))^{1/2s}. \]
Choosing $R$ to be some sufficiently large real number in terms of $\epsilon,s,|\mathcal{A}|, \xi_0, \dots, \xi_{2s},$ say, $R \geq \epsilon^{-1} (4s^2|\mathcal{A}|)^{4s^2}  \sum_{i=0}^{2s}  \xi_{i}^{-1},$ supplies the desired inequality.
\end{proof}

We note that such results can also be extended for multiplicative energies of sets of positive real numbers. In order to see this, note that for any finite set $A \subseteq (0, \infty)$, we can consider the set $\phi(A)$, where $\phi(x) = \log x$. In this case, we have
\[ x_1 \dots x_s = x_{s+1} \dots x_{s+r} \ \ \text{if and only if} \ \ \phi(x_1) + \dots + \phi(x_s) = \phi(x_{s+1}) + \dots + \phi(x_{s+r}), \]
for every $s,r \in \mathbb{N}$ and for every $x_1, \dots, x_{s+r} \in A$. Consequently, we see that $M_{s}(X) = E_{s}(\phi(X))$ for every $s \in \mathbb{N}$ and for every $X \subseteq A$. Similarly, for every $m,n \in \mathbb{N} \cup \{0\}$ and for every $X\subseteq A$, we have $|X^{(m)}/X^{(n)}| = |m \phi(X) - n\phi(X)|$. 
\par

\begin{lemma} \label{mlpain}
Let $A_1, \dots, A_{2s}$ be finite subsets of $\mathbb{R} \setminus \{0\}$. Then 
\[ M_{s}(A_1, \dots, A_{2s}) \leq 2^{2s} M_{s}(A_1)^{1/2s} \dots M_{s}(A_{2s})^{1/2s}. \]
\end{lemma}

\begin{proof}
Writing $A_{i} = (-A_{i,1}) \cup A_{i,2}$ where $A_{i,1}, A_{i,2} \subseteq (0 , \infty)$ for every $1 \leq i \leq 2s$, we have 
\begin{align} \label{deshor}
M_{s}(A_1, \dots, A_{2s}) &= \sum_{1 \leq j_1 , \dots, j_{2s} \leq 2} M_{s}((-1)^{j_1} A_{1,j_1}, \dots, (-1)^{j_{2s}} A_{2s,j_{2s}}) \nonumber \\
& \leq \sum_{1 \leq j_1 , \dots, j_{2s} \leq 2} M_{s}( A_{1,j_1}, \dots,  A_{2s,j_{2s}}) , 
\end{align}
where the last inequality follows from the fact that whenever $x_1 \dots x_{s} = x_{s+1} \dots x_{2s}$, then $|x_1| \dots |x_{s}| = |x_{s+1}| \dots |x_{2s}|$. As before, writing $\phi(X) = \{ \log x \ | \ x \in X\}$ for every finite set $X \subseteq (0, \infty)$ and consequently applying Lemma $\ref{yoc2}$, we discern that 
\begin{align*}
 M_{s}( A_{1,j_1}, \dots,  A_{2s,j_{2s}}) & = E_{s}(\phi( A_{1,j_1}), \dots, \phi(A_{2s,j_{2s}})) \leq \prod_{i=1}^{2s} E_{s}(\phi(A_{i,j_i}))^{1/2s}  \\
& = \prod_{i=1}^{2s} M_{s}(A_{i,j_i})^{1/2s} \leq \prod_{i=1}^{2s} M_{s}(A_i)^{1/2s},
\end{align*}
for every $1 \leq j_1, \dots, j_{2s} \leq 2$. Combining this with $\eqref{deshor}$ finishes our proof.
\end{proof}

As in the case of Theorem $\ref{zidt}$, the condition $A_{i} \subseteq \mathbb{R} \setminus \{0\}$ is necessary in the above lemma, since we may choose $A_1 = A_{s+1} = \{0\}$ and $A_{i} = A_{s+i} = P_N$ for every $2 \leq i \leq s$, where $P_N$ is the set of first $N$ primes. In this case, since we have $M_{s}(P_N) \ll_{s} N^{s}$, we get $2^{2s} M_{s}(A_1)^{1/2s} \dots M_{s}(A_{2s})^{1/2s} \ll_{s} N^{s-1}$, while $M_{s}(A_1, \dots, A_{2s}) \geq N^{2s- 2}$ due to the fact that $0 \cdot x_2 \dots x_{s} = 0 \cdot x_{s+2} \dots x_{2s}$ for any $x_2, \dots, x_{s}, x_{s+2}, \dots, x_{2s} \in \mathbb{R}$. 


\begin{lemma} \label{hld3}
Let $A_1, \dots, A_n$ be pairwise disjoint finite sets of real numbers, and let $A_0 = A_1 \cup \dots \cup A_n$. Then 
\[ E_{s}(A_0) \leq n^{2s-1} \sum_{i=1}^{n} E_{s}(A_i)  \leq n^{2s} \sup_{1 \leq i \leq n} E_{s}(A_i) . \]
Moreover, if $0 \notin A_0$, then
\[ M_{s}(A_0) \leq 2^{2s} n^{2s-1} \sum_{i=1}^{n} M_{s}(A_i)  \leq (2n)^{2s} \sup_{1 \leq i \leq n} M_{s}(A_i) . \]
\end{lemma}
\begin{proof}
A straightforward application of Lemma $\ref{yoc2}$ and H\"{o}lder's inequality implies that
\begin{align*}
E_{s}(A_0) 
& = \sum_{1 \leq i_1, \dots, i_{2s} \leq n} E_{s}(A_{i_1}, \dots, A_{i_{2s}}) \leq \sum_{1 \leq i_1, \dots, i_{2s} \leq n} \prod_{j=1}^{2s} E_{s}(A_{i_j})^{1/2s} \\
& = \big(\sum_{i=1}^{n} E_{s}(A_i)^{1/2s}\big)^{2s} \leq n^{2s-1} \sum_{i=1}^{n} E_{s}(A_i) \leq n^{2s} \sup_{1 \leq i \leq n} E_{s}(A_i),
\end{align*}
whence we obtain the first inequality stated in Lemma $\ref{hld3}$. The proof of the second inequality follows mutatis mutandis, except we apply Lemma $\ref{mlpain}$ instead of Lemma $\ref{yoc2}$. 
\end{proof}

Our proof of Proposition $\ref{kp}$ will be using the Balog--Szemer\'edi--Gowers theorem, that allows us to convert information on $2$-fold additive energies into bounds for the sumset. The former hypothesis can equivalently be expressed in terms of estimates on restricted sumsets on graphs, and we present one such variant here \cite[Theorem $5$]{Bal2007}.

\begin{lemma} \label{balbsg}
Let $A, B$ be finite subsets of an additive abelian group $Z$ and let $G$ be a subset of $A \times B$ with $S = \{ a+b \ | \ (a,b) \in G\}$. If $|A|, |B|, |S| \leq N$ and $|G| \geq \alpha N^2$, then there exists an $A' \subseteq A$ such that
\[ |A'+A'| \leq \frac{2^{38}}{3} \frac{\log (32/\alpha)}{ \alpha^{7}} N \ \ \text{and} \  \ |A'| \geq \frac{3}{2^{16}}  \frac{\alpha^3}{\log(32/\alpha) } N  \]
\end{lemma}

We will also be using a classical result in additive combinatorics that bounds the size of many-fold sumsets in terms of the cardinality of two-fold sumsets. Thus, we record the Pl{\"u}nnecke--Ruzsa theorem \cite[Corollary 6.29]{TV2006}.
\begin{lemma} \label{pr21}
Let $A$ be a finite subset of some additive abelian group $G$. If $|A+A| \leq K|A|$, then for all non-negative integers $m,n$, we have
\[  |mA - nA| \leq K^{m+n}|A|. \]
\end{lemma}

We now present further details regarding the aforementioned work of Bourgain--Chang, and we begin this endeavour by defining, for each finite set $A$ of integers and for every real number $q \geq 1$, the quantity
\[ \lambda_{q}(A) = \sup \bigg(\int_{[0,1)}| \sum_{n \in A} c_n e(n\theta) |^{q} d\theta \bigg)^{1/q}, \]
where the supremum is being taken over all sequences $\{ c_n\}_{n \in A}$ satisfying $\sum_{n \in A} |c_n|^2 \leq 1.$ We remark that estimates on $\lambda_{q}(A)$ imply bounds for additive energies of subsets of $A$, namely, we can use orthogonality to infer that for every non-empty subset $B$ of $A$, we have
\[ |B|^{-s} E_{s}(B) = \int_{[0,1)} \big|\sum_{n \in B} |B|^{-1/2} e(n \theta)\big|^{2s} d\theta \leq \lambda_{2s}(A)^{2s}, \]
whence, we get $E_{s}(B) \leq \lambda_{2s}(A)^{2s} |B|^s.$ We now record the following result of Bourgain--Chang \cite[Proposition $2$]{BC2005} which produces upper bounds for $\lambda_{q}(A)$ whenever $A$ has a small product set. 

\begin{lemma} \label{bourch}
Let $\gamma >0$ and $q > 2$. Then there exists a constant $\Lambda = \Lambda(\gamma, q)$ such that if $A$ is a finite set of integers and $K$ is some real number satisfying $|A \cdot A| \leq K |A|$, then
\[ \lambda_{q}(A) < K^{\Lambda} |A|^{\gamma}. \]
\end{lemma}

In fact, this bound has recently been quantitatively strengthened \cite{PZ2020}, wherein, it is now known that when $\gamma < \log q$, one may choose $\Lambda = 6(1 + \frac{\log q}{\gamma})$. In order to see this, observe that  a combination of \cite[Theorem $1.3$]{PZ2020} and Lemma $\ref{pr21}$ yields the bound
\[ \lambda_{q}(A) \leq (|A^{(3)}| |A|^{-1})^{2 + 1/\epsilon} |A|^{2\epsilon \log q} \leq (|A\cdot A| |A|^{-1})^{6 + 3/\epsilon} |A|^{2\epsilon \log q}, \]
for every $0 <\epsilon < 1/2$. We may now set $\epsilon = \gamma/ (2 \log q)$ to confirm our claim.
\par
%

Finally, we note a straightforward corollary from the work of Hanson, Roche-Newton and Rudnev \cite{HRR2020}, which will play an important role in our proof of Theorem $\ref{rtp}$.
\begin{lemma} \label{hrr5}
Let $k \geq 2$ be a natural number, let $\mathcal{I} \subseteq \mathbb{R}$ be an interval, let $f: \mathcal{I} \to \mathbb{R}$ be $(k-1)$-convex on $\mathcal{I}$, let $I \subseteq \mathcal{I}$ be a finite set, with $A = f(I)$ and $K = |2I-I|/|I|$. Moreover, let $A' \subseteq A$ such that $|A'| \geq C |A|$. Then 
\[ |2^{k-1}A' - (2^{k-1}-1)A'| \gg_k |A'|^{k} (C/K)^{2^{k} - k -1} (\log|A|)^{-2^{k+1} - k-3}. \]
\end{lemma}

\begin{proof}
Let $I'$ be the subset of $I$ such that $A' = f(I')$. In this case, we see that $|I'| \geq C |I|$, whence, $|2I'- I'| \leq K |I| \leq K C^{-1} |I'|.$ Applying \cite[Theorem $1.4$]{HRR2020} yields the desired conclusion.
\end{proof}


\section{Set up for the decomposition}

We will use this section to prove two results that will perform a key role in our proofs of Theorems $\ref{gemn}$ and $\ref{eric}$. Our first result in this section encapsulates the Balog--Szemer\'edi--Gowers type philosophy, that is, given any finite set $A$ of real numbers, either we should expect a power saving over the trivial bound for $M_{s}(A)$, or there must exist a large subset $B$ of $A$ such that the many-fold product sets of $B$ are suitably small.

\begin{lemma} \label{ntm2}
Let $k \geq 1$ be a real number, let $m \geq k$ be a natural number. Let $U \geq 120 m$ and $s \geq 2^{5 + ( 1 + U) (\ceil{\log k}+1)}$ be some natural numbers and let $A$ be a finite set of positive real numbers. Then either $M_{s}(A) < |A|^{2s - k}$ or there must exist $B \subseteq A$ such that 
\begin{equation} \label{imb}
 |B| \geq C |A|^{1 - 82k/U}  \ \text{and} \ |B^{(m)} |  \leq C_m |A|^{3k},
 \end{equation}
 for some absolute constants $C, C_m > 0$. 
\end{lemma}
\begin{proof}
We begin by noting that it is sufficient to consider the case when $s=s_0$, where $s_0 = 2^{5 + ( 1 + U) (\ceil{\log k}+1)}$, since whenever $s$ is strictly larger, we may use Lemma $\ref{yoc}$ to derive the inequality $M_{s}(A) \leq |A|^{2s- 2s_0} M_{s_0}(A)$, which then allows us to infer the desired result from the conclusion derived in the $s=s_0$ case. Thus setting $s=s_0$, we may further assume that $M_{s}(A) \geq |A|^{2s-k}$, since otherwise we are done. In this case, noting the discussion following Lemma $\ref{yoc2}$, we may apply a multiplicative version of Proposition $\ref{kp}$ to deduce that either
\[  M_{s/2}(A) > |A|^{s- k+ k/U}, \]
or there exists some $B \subseteq A$ such that 
\begin{equation} \label{lpzee}
 |B| \geq C |A|^{1- 82k/U} \ \text{and} \ |B^{(m)}| \leq C_m|A|^{k (1+ 240 m /U)} \leq C_m |A|^{3k}.
 \end{equation}
Since the latter conclusion would mean that we are done, we may assume that $M_{s/2}(A) > |A|^{s - k(1-1/U)}.$ We can now iterate this argument multiple times to deduce that
\[ M_{s/2^{r}}(A) > |A|^{s/2^{r-1} - k(1-1/U)^{r}}, \]
for every $r$ satisfying $2^r \leq s/4$. Upon comparing this with the trivial bound
\[ M_{s/2^r}(A) \leq |A|^{s/2^{r-1} -1}, \]
we infer that $k(1-1/U)^{r} \geq 1,$ that is, 
\[ r \leq \frac{\log k}{\log (\frac{U}{U-1})} < \frac{\log k}{\log (\frac{U+1}{U})} \leq  (1 + U ) \log k, \]
where the last step utilises the fact that $\log(1+x)> \frac{x}{x+1}$ whenever $x>0$. Since $s \geq 2^{5 + ( 1 + U) (\ceil{\log k}+1)}$, we may choose $r = 2  +  ( 1 + U) (\ceil{\log k}+1)$ to obtain a contradiction, whence, there must exist a set $B \subseteq A$ satisfying $\eqref{lpzee}$, and so, we are done.
\end{proof}

In the forthcoming sections, we will use Lemma $\ref{ntm2}$ to deduce inverse results that may then be iteratively applied to obtain a decomposition of $A$ into a sequence of subsets, such that each such subset either has a small multiplicative energy or a small additive energy. An important aspect of such algorithmic arguments is to ensure that they finish in an appropriate number of steps, and in this endeavour, we present the following lemma.


\begin{lemma} \label{com2}
Let $0 < c < 1$ and $C>0$ be constants. Let $A_0 = A$, and for each $i \geq 1$, define $A_i = A_{i-1} \setminus U_i$ where $|U_i| \geq C|A_{i-1}|^{1- c}$. Then, for some $r \leq 2(\log |A| + 2) +  C^{-1} \frac{ |A|^{c}}{2^c  - 1} $, we must have $|A_r| \leq 1$. 
\end{lemma}
\begin{proof}
Let $j$ be the smallest natural number such that $|A_j| \leq |A|/2$. Then, for each $1 \leq i \leq j$, we see that $|U_i| \geq C|A_{i-1}|^{1-c} \geq C|A|^{1-c} 2^{c-1},$ and so, we have
\[ |A|/2 \leq |A_{j-1}| = |A| - \sum_{i=1}^{j-2} |U_i| \leq |A| - C(j-2)|A|^{1-c} 2^{c-1}. \]
Upon simplifying the above inequality, we see that $j \leq 2 + C^{-1}2^{-c} |A|^{c},$ whence, we can proceed with at most 
\[  \sum_{n=0}^{\log |A| +1} (2 + C^{-1}2^{-c} |A|^{c} 2^{-nc}) 
\leq 2(\log |A| + 2) +  C^{-1} \frac{ |A|^{c}}{2^c  - 1}  \]
number of steps before $|A_i| \leq 1$, for some $i \in \mathbb{N}$.
\end{proof}


\section{Proof of Theorem $\ref{gemn}$}



We dedicate this section to proving Theorem $\ref{gemn}$ and Corollary $\ref{bta}$. Moreover, for the purposes of this section, we will fix $k\geq 1$ to be some real number, $q$ to be some even natural number, and we will let $\gamma = 1/4 - 1/100$. With $q, \gamma$ fixed, we will write $\Lambda = \Lambda(q, \gamma)$ to be the constant arising in the conclusion of Lemma $\ref{bourch}$. Noting the discussion following Lemma $\ref{bourch}$, we may choose $\Lambda = 6(1 + {\log q}/{\gamma}) = 6 + 25 \log q$. Finally, we will set $l = \ceil{600qk\Lambda}$, and $m = 2^l$, and $U = 120m$, and $s =2^{5 + (1+U) (\ceil{\log k }+1)}$.
\par


As mentioned in \S3, we begin by presenting a lemma that allows us to show that any set with a sizeable amount of multiplicative energy contains a suitably large subset with a small additive energy.

\begin{lemma} \label{hop2}
Let $A$ be a non-empty, finite subset of natural numbers such that $M_{s}(A) \geq |A|^{2s - k}.$ Then, there exists a subset $B$ of $A$ such that $|B| \geq C |A|^{1 - 82k/U}$ and 
 \[ E_{q/2}(B) \ll_{k,q}  |B|^{q-q/4}.\]
\end{lemma}

\begin{proof}
We note that our choice of $k,m,U$ and $s$ satsify the hypothesis of Lemma $\ref{ntm2}$, and so, we may apply the same to obtain a subset $B$ of $A$ satisfying $\eqref{imb}$. In particular, this means that 
\[  \prod_{i=0}^{l-1} \frac{|B^{(2^{i+1})}| }{|B^{(2^i)}|} =\frac{ |B^{(m)}| }{|B|} \leq C_m \frac{|A|^{3k}}{|B|}, \]
whence, there exists some $0 \leq i \leq l-1$, such that
\[  \frac{|B^{(2^i)} \cdot B^{(2^i)}|}{|B^{(2^i)}|} = \frac{|B^{(2^{i+1})}| }{|B^{(2^i)}|} \leq C_m^{1/l} \frac{|A|^{3k/l}}{|B|^{1/l}}. \]
We now use Lemma $\ref{bourch}$ to deduce that 
\[ \lambda_{q}(B^{(2^i)}) <  C_m^{\Lambda/l} \frac{|A|^{3k\Lambda/l}}{|B|^{\Lambda/l}} |B|^{\gamma}, \]
which, as per the discussion preceding Lemma $\ref{bourch}$, implies that for every non-empty subset $D$ of $B^{(2^i)}$, we have
\[ E_{q/2}(D) \leq C_m^{q\Lambda/l}  \frac{|A|^{3qk\Lambda/l}}{|B|^{q\Lambda/l}} |B|^{q\gamma} |D|^{q/2}. \]
Note that there exists some natural number $g$ such that $g \cdot B \subseteq B^{(2^i)}$, and since additive energies remain invariant under affine transformations, we see that
\[ E_{q/2}(B) = E_{q/2}(g \cdot B) \leq  C_m^{q\Lambda/l}  \frac{|A|^{3qk\Lambda/l}}{|B|^{q\Lambda/l}} |B|^{q\gamma} |B|^{q/2} = C_m^{q\Lambda/l}  |A|^{3qk\Lambda/l} |B|^{q(\gamma+1/2-\Lambda/l )}. \]
Recalling that $|B| \geq C |A|^{1 - 82k/U}$, we get
\[  E_{q/2}(B) \leq \frac{C_m^{q\Lambda/l}}{C^{(1 - 82k/U)^{-1}} }|B|^{\frac{3qk\Lambda}{l (1 - 82k/U)}} |B|^{q(\gamma+1/2-\Lambda/l )}. \]
Furthermore, using the fact that $(1-x)^{-1} \leq 1 + 2x$ whenever $x \leq 1/2$, we see that
\[ \frac{3qk\Lambda}{l (1 - 82k/U)} \leq 3qk \frac{\Lambda(1 + 164k/U)}{l}. \]
Noting that $164k/U \leq 1/10$ and $l \geq 600qk \Lambda$, we deduce that
\[ E_{q/2}(B) \leq  \frac{C_m^{q\Lambda/l}}{C^{(1 - 82k/U)^{-1}} } |B|^{1/100} |B|^{3q/4 - q/100} \ll_{k,q} |B|^{q-q/4} ,\]
which is the claimed estimate. 
\end{proof}


We are now ready to proceed with our proof of Theorem $\ref{gemn}$.
\begin{proof}[Proof of Theorem $\ref{gemn}$]

We begin by noting that it suffices to prove Theorem $\ref{gemn}$ for finite subsets of natural numbers. In order to see this, let $A$ be a finite subset of integers and write $A = A_1 \cup A_2 \cup A_3$, where $A_1 \subseteq [1, \infty)$ and $A_2 \subseteq (-\infty, -1]$ and $A_3 \subseteq \{0\}$. Note that whenever $A_3$ is non-empty, we trivially have $M_{s}(A_3) = E_{s}(A_3) \leq |A|^{s}$. On the other hand, applying Theorem $\ref{gemn}$ for the set $A_1$, we see that $A_1 = B_1 \cup C_1$ such that 
\[ E_{q/2}(B_1) \ll_{k,q} |B_1|^{q-q/5} \ \text{and} \ M_{s}(C_1) \leq |C_1|^{2s - k}. \]
Similarly, since $-A_2 \subseteq \mathbb{N}$, and $E_{s}(-X) = E_{s}(X)$ and $M_{s}(-X) = M_{s}(X)$ for every finite subset $X$ of real numbers, we may apply Theorem $\ref{gemn}$ to deduce that $A_2 = B_2 \cup C_2$ where $B_2$ and $C_2$ satisfy the relevant energy estimates. We now set $B = B_1 \cup B_2 \cup A_3$ and use Lemma $\ref{hld3}$ to furnish the bound 
\[ E_{q/2}(B) \ll_{k,q} |B|^{q-q/5}.\]
Similarly, we may write $C= C_1 \cup C_2$ and use Lemma $\ref{hld3}$ to obtain the estimate
\[ M_{s}(C) \ll_{s} |C|^{2s - k}. \]
Thus, from this point on, we will assume that $A$ is a finite subset of natural numbers. 
\par

We begin by running an algorithm to procure sets $B$ and $C$ with desirable arithmetic properties. Thus, we set $A_0 = A$ and $B_0 = \emptyset$. Moreover, for every $i \geq 1$, at the beginning of the $i^{th}$ iteration, we will assume that we have two disjoint sets $A_{i-1}$ and $B_{i-1}$ such that $A_{i-1} \cup B_{i-1} = A$. If $M_{s}(A_{i-1}) \leq |A_{i-1}|^{2s-k}$, we stop our algorithm. On the other hand, if
\[ M_{s}(A_{i-1}) > |A_{i-1}|^{2s-k}, \]
we apply Lemma $\ref{hop2}$ to deduce the existence of $D_i \subseteq A_{i-1}$ such that $|D_i| \geq C |A_{i-1}|^{1 - 82k/U}$ and 
 \begin{equation} \label{lpt2}
  E_{q/2}(D_i) \ll_{k,q}  |D_i|^{q-q/4}.
  \end{equation}
We write $A_i = A_{i-1} \setminus D_i$ and $B_i = B_{i-1} \cup D_i$, and proceed to commence the $(i+1)^{th}$ iteration.
\par

As per Lemma $\ref{com2}$, such an algorithm can run for at most $r$ steps, where $r$ is some natural number satisfying $r \leq 2(\log |A| + 2) +  C^{-1} \frac{ |A|^{c}}{2^c  - 1}$ and $c = 82k/U$. Setting $B = B_{r}$ and $C = A_{r}$, we see that
\[ M_{s}(C) \leq |C|^{2s - k}. \]
Moreover, if $C \neq A$, that is, if $B$ is non-empty, we have $B = D_1 \cup \dots \cup D_r$ where each $D_i$ satisfies $\eqref{lpt2}$. Combining Lemma $\ref{hld3}$ along with $\eqref{lpt2}$, we get
\[ E_{q/2}(B)  \leq r^{q} \sup_{1 \leq i \leq r} E_{q/2}(D_i)  + O(1) \ll_{k,q} r^{q} |D_i|^{q-q/4}.\]
Finally, noting our upper bounds for $r$, we deduce that
\begin{equation} \label{sum9}
E_{q/2}(B) \ll_{k,q} |A|^{82qk/U}|B|^{q-q/4}. 
\end{equation}
We further observe that $|D_1| \geq |A|^{1- 82k/U}$, whence, $|A|^{\frac{82qk}{U}} \leq |D_1|^{\frac{82qk}{U- 82k}}$. Additionally, since $U \geq 120 \cdot 2^{600qk \Lambda}$, where $\Lambda \geq 1$, we see that
\[ \frac{82qk}{U- 82k} \leq \frac{164qk}{U} \leq q/20. \]
Combining this with $\eqref{sum9}$ and the fact that $|D_1| \leq |B|$, we discern that 
\[ E_{q/2}(B) \ll_{k,q} |B|^{q - q/5}, \]
and so, we are done.
\end{proof}

We end this section by showing how Corollary $\ref{bta}$ and Theorem $\ref{zidt}$ may be deduced by combining our results from \S2 along with Theorem $\ref{gemn}$.

\begin{proof}[Proof of Corollary $\ref{bta}$]
Since $s$ is sufficiently large, there exists a real number $k \geq 4$ such that upon setting 
\[  q = 10\ceil{k}, \ \text{and} \ \Lambda =6+25 \log q, \ \text{and} \  U = 120 \cdot 2^{\ceil{600qk\Lambda}},\ \text{and} \ s_1 = 2^{5 + (1+U) (\ceil{\log k }+1)}, \]
we have $\log \log s_1 \gg \log \log s \geq \log \log s_1$. Moreover, our choice of $k,q, \Lambda$ and $s_{1}$ satisfy the hypothesis of Theorem $\ref{gemn}$, whence, there exist pairwise disjoint set $B, C$ such that $A = B \cup C$ and
\[ E_{q/2}(B) \ll_{k,q} |B|^{q-q/5} \leq |B|^{q- k}, \ \text{and} \ M_{s_1}(C) \ll_{s_1} |C|^{2s_1 - k}. \]
We may now use Lemma $\ref{yoc}$ along with the fact that $s \geq s_1$ to deduce that
\[ E_{s}(B) \ll_{k,q} |B|^{s - k} \ \text{and} \ M_{s}(C) \ll_{s} |C|^{2s - k}. \]
Thus, it suffices to show that $k \gg (\log \log s)^{1/2}(\log \log \log s)^{-1/2}$, and this follows from noting that
\[ \log \log s \ll \log \log s_1\ll \log (2U) + \log (\log k +1) \ll kq \log q \ll k^2 \log k , \]
whenceforth, we obtain the claimed estimate.
\end{proof}

\begin{proof}[Proof of Theorem $\ref{zidt}$]
We apply Corollary $\ref{bta}$ to deduce the existence of pairwise disjoint sets $B$ and $C$ such that $A = B \cup C$ and
\begin{equation} \label{doitagain}
E_{s}(B) \ll_{s} |B|^{2s - \eta_{s}} \ \text{and} \ M_{s}(C) \ll_{s} |C|^{2s - \eta_{s}},
\end{equation}
where $\eta_{s} \gg (\log \log s)^{1/2} (\log \log \log s)^{-1/2}.$ Moreover, given any set $X$ of real numbers and any $n \in \mathbb{R}$, we write $q_{s}(X;n) = |\{ (x_1, \dots, x_s) \in X^s \ | \ x_1 \dots x_s = n\}|$. With this notation in hand, we see that
\[ M_{s}(B,C) = \sum_{n \in \mathbb{Z}} q_{s}(B;n) q_{s}(C;n), \]
whence, we may apply the Cauchy-Schwarz inequality to deduce that
\begin{align*}
 M_{s}(B,C) 
& \leq  \big( \sum_{n \in \mathbb{Z}} q_{s}(B;n)^2) \big)^{1/2} \big( \sum_{n \in \mathbb{Z}} q_{s}(C;n)^2) \big)^{1/2} = M_{s}(B)^{1/2} M_{s}(C)^{1/2}.  
\end{align*}
Combining this with $\eqref{doitagain}$ and the trivial bound $M_{s}(B) \leq |B|^{2s-1}$, we obtain the required estimate for $M_{s}(B,C)$, with $\gamma_{s} = (\eta_{s}+ 1)/2$. Furthermore, the case of $E_{s}(B,C)$ can be resolved mutatis mutandis.
\par

Finally, given any $A_1, \dots, A_{2s} \in \{B,C\}$, at least $s$ of these sets must be the same, and so, suppose that $A_1 = \dots = A_{s}$. If $A_1 = \dots = A_{s} = C$, Lemma $\ref{mlpain}$ implies that
\[ M_{s}(A_1, \dots, A_{2s}) \ll_{s} M_{s}(C)^{1/2} \prod_{i=s+1}^{2s} (|A_i|^{2s-1})^{1/2s} \ll_{s} |A|^{2s - \eta_{s}/2 -1/2}, \]
while if $A_1 = \dots = A_{s} = B$, we can apply Lemma $\ref{yoc2}$ to discern that
\[ E_{s}(A_1, \dots, A_{2s}) \leq E_{s}(B)^{1/2} \prod_{i=s+1}^{2s} ( |A_i|^{2s-1})^{1/2s} \ll_{s} |A|^{2s - \eta_{s}/2 - 1/2}. \qedhere\]
\end{proof}


\section{Proof of Theorem $\ref{eric}$}

We use this section to prove Theorem $\ref{eric}$, and so, let $(\mathcal{Z}, m,b)$ be a good tuple. Here, note that $m \geq b$ trivially, and furthermore, we will assume that $b \geq 30$, since otherwise, we may use trivial bounds of the shape 
 \begin{equation} \label{shaday}
 E_{s}(A) \leq |A|^{2s-1}, 
 \end{equation}
to finish our proof.

\begin{lemma} \label{thsam}
Let $k = b/30$, let $U_2 = 120 m$, let $s_2 = 2^{5 + (1+U_2) (\ceil{\log k} + 1)}$, let $U_1 = 500\ceil{k}s_2$ and let $s_1 = 2^{5 + (1+U_1) (\ceil{\log k} + 1)}$. Then for every finite, non-empty subset $A$ of $\mathcal{Z}$, we have either $E_{s_1}(A) < |A|^{2s_1-k}$, or there exists some subset $A_1 \subseteq A$ such that 
\[ |A_1| \geq C |A|^{1 - 82k/U_1}  \ \text{and} \ M_{s_2}(A_1) < |A_1|^{2s_2-k}, \]
or $|A| \leq \mathcal{C}_{k, m}$, where $\mathcal{C}_{k,m}$ is some positive constant. 
\end{lemma}

\begin{proof}
Let $A$ be a non-empty finite subset of $\mathcal{Z}$. We may suppose that we have $E_{s_1}(A) \geq |A|^{2s_1-k}$, since otherwise, we are done. Furthermore, we may apply arguments as in the proof of Lemma $\ref{ntm2}$ to find a subset $A_1$ of $A$ such that
\[ |A_1| \geq C |A|^{1 - 82k/U_1}  \ \text{and} \ |mA_1 |  \leq C_{m} |A|^{3k}.\]
Here, we have used the facts that $s_1 \geq 2^{5 + (1+U_1) (\ceil{\log k_1} + 1)}$ and $U_1 \geq 120 m$. We can now assume that $M_{s_2}(A_1) \geq |A_1|^{2s_2 - k}$, since otherwise we would be done, whence, as before, we obtain a subset $A_2$ of $A_1$ such that
\[ |A_2| \geq C |A_1|^{1 - 82k/U_2}  \ \text{and} \ | A_2^{(m)} |  \leq C_{m} |A_1|^{3k}. \]
This implies that 
\[ |A_2^{(m)}| \leq C_{m} C^{(1-82k/U_2)^{-1}} |A_2|^{\frac{3k}{1 - 82k/U_2}}, \]
and
\[ |mA_2| \leq |mA_1| \leq C_{m} |A|^{3k} \ll_{k,m} |A_2|^{\frac{3k}{ (1- 82k/U_1)(1- 82k/U_2)}} . \]
Noting the elementary inequality $(1+x)^{-1} \leq 1+2x$ whenever $0 \leq x \leq 1/2$, and the fact that $U_2 \geq U_1 \geq 120k$, we see that
\[ |A_2^{(m)}| \ll_{k,m} |A_2|^{9k},  \ \text{and} \ |mA_2| \ll_{k,m} |A_2|^{27k}. \]
Putting this together with $\eqref{hypco}$ and the fact that $k = b/30$, we discern that
\[ |A_2|^{b} \ll_{k,m} |A_2|^{9b/10}, \]
from which, we can infer that $|A_2|  \ll_{k,m} 1$, and consequently, we have $|A| \ll_{k,m} 1$, thus finishing our proof of Lemma $\ref{thsam}$.
\end{proof}

\begin{proof}[Proof of Theorem $\ref{eric}$]
Let $A$ be a finite, non-empty subset of $\mathcal{Z}$. As in the proof of Theorem $\ref{gemn}$, it suffices to consider the case when $A \subseteq (0, \infty)$. Thus, assuming the aforementioned condition, we proceed algorithmically, and so, we set $A_0 = A$. At the start of the $i^{th}$ iteration, we will have two disjoint sets $A_{i-1}$ and $B_{i-1}$ such that $A = A_{i-1} \cup B_{i-1}$. In case $E_{s_1}(A_{i-1}) < |A|^{2s_1-k}$ or $|A_{i-1}| \leq \mathcal{C}_{k,m}$, we stop the iteration, else, we apply Lemma $\ref{thsam}$ to deduce the existence of some subset $D_i \subseteq A_{i-1}$ such that
\begin{equation} \label{sherf}
|D_i| \geq C |A_{i-1}|^{1 - 82k/U_1}  \ \text{and} \ M_{s_2}(D_i) < |D_i|^{2s_2-k}. 
\end{equation}
In this case, we set $A_i = A_{i-1} \setminus D_{i}$ and $B_i = B_{i-1} \cup D_i$, and proceed to commence the $(i+1)^{th}$ iteration.
\par

Noting Lemma $\ref{com2}$, we see that such an algorithm can run for at most $r$ steps, where 
\begin{equation} \label{rstop}
 r \leq 2(\log |A| +2)  + C^{-1} \frac{|A|^{82k/U_1}}{2^{82k/U_1}-1}. 
 \end{equation}
Thus, we must have $A = A_r \cup B_r$, where either $E_{s_1}(A_{r}) < |A_r|^{2s_1-k}$ or $|A_{r}| \leq \mathcal{C}_{k,m}$, and consequently, we find that
\[ E_{s_1}(A_r) \ll_{k,m} |A_r|^{2s_1- k}. \]
Moreover, if the set $B_r$ is non-empty, then $B_r = D_1 \cup \dots \cup D_r,$ where $D_i$ satisfies $\eqref{sherf}$ for every $1 \leq i \leq r$. Combining inequalities $\eqref{sherf}$ and $\eqref{rstop}$ with Lemma $\ref{hld3}$, we see that
\[ M_{s_2}(B_r) \leq (2r)^{2s_2} \sup_{1 \leq i \leq r} M_{s_2}(D_i) \ll_{k,m} C^{-2s_2} |A|^{164ks_2/U_1} |B_r|^{2s_2-k}.\]
Recalling that $U_1 = 500 \ceil{k} s_2$, we see that
\[ |A|^{164k s_2/U_1} \ll_{k,m} |D_1|^{(164k s_2/U_1)(1+ 164 k/U_1)} \ll_{k,m} |B_r|^{492 k s_2/U_1} \ll_{k,m} |B_r|,\]
whence, we have
\[ M_{s_2}(B_r) \ll_{k,m} |B_r|^{2s_2 - k +1}. \]
Setting $B= B_r$ and $C = A_r$ finishes the proof of Theorem $\ref{eric}$. 
\end{proof}

%


\section{Proof of Theorem $\ref{rtp}$}

In this section, we present our proof of Theorem $\ref{rtp}$. Thus, let $k \geq 2$ be a natural number, let $\mathcal{I} \subseteq \mathbb{R}$ be an interval, let $f: \mathcal{I} \to \mathbb{R}$ be $(k-1)$-convex on $\mathcal{I}$ and let $I \subseteq \mathcal{I}$ be a finite set, with $A = f(I)$ and $K = |2I-I|/|I|$. In this case, our main aim is to prove that the estimate 
\[ E_{s}(A) \ll_{s,k} |A|^{2s - k + (4k-4)s^{-\eta_k}} + K^{O_{s,k}(1)}\]
 holds for every natural number $s$, where $\eta_k = \log (1 + T_k^{-1})$ with 
 \[ T_k = 2( 82 k + 82 \cdot (2^{k} - k - 1) + 240 \cdot 2^k). \]
  Since $\eta_k \in (0,1)$, we see that whenever $s \leq 3$, we may use the trivial estimate $\eqref{shaday}$ to prove Theorem $\ref{rtp}$, whence, we may assume that $s \geq 4$. In fact, we begin by focusing on the case when $s= 2^{r+1}$ for some $r \in \mathbb{N}$, and so, we record the following lemma.

\begin{lemma} \label{thrt}
Let $s= 2^{r+1}$ for some $r \in \mathbb{N}$. Then either $E_{s}(A) \ll_{s,k} |A|^{2s - k + (2k-2) s^{-\eta_k}}$ or $|A| \ll_{s,k} K^{O_{s,k}(1)}$. 
\end{lemma}
\begin{proof}
Writing $E_{s}(A) = |A|^{2s - k + \Lambda}$, we may assume that $\Lambda > (k-1) s^{-\eta_k}$, since otherwise, we are done. We first study the case when $E_{s/2}(A) \leq |A|^{s- k + \Lambda + \delta}$, where $\delta=\Lambda T_k^{-1}$. Setting $m = n+1 = 2^{k-1}$, we use Proposition $\ref{kp}$ and Lemma $\ref{hrr5}$ to deduce the existence of some set $A' \subseteq A$ such that $|A'| \gg |A|^{1- 82 \delta}$, and 
\[  |mA'-nA'| \gg_{k} |A|^{k-82 \delta k} ( |A|^{-82 \delta}/K)^{2^{k} - k -1} (\log |A'|)^{-2^{k+1} - k-3} , \]
and 
\[ |mA'-nA'| \ll_{k} |A|^{k- \Lambda + 240 \cdot 2^k \delta}. \]
Putting these together, we see that
\[ |A|^{ \Lambda - \delta( 82 k + 82 \cdot (2^{k} - k - 1) + 240 \cdot 2^k)} \ll_{k} (\log |A|)^{2^{k+1} + k+3}K^{2^k -  k-1}. \]
As $\delta = \Lambda T_k^{-1}$, we see that the left hand side in the above inequality is at least $|A|^{\Lambda/2}$. Moreover, since $\Lambda > (k-1)s^{-\eta_k}$ and $\log |A| \ll_{\epsilon} |A|^{\epsilon}$ for every $\epsilon >0$, we may deduce that
\[ |A| \ll_{s,k} K^{4s^{\eta_k}(2^k - k - 1)(k-1)^{-1}},\]
and so, Lemma $\ref{thrt}$ holds true in this case.
\par

Hence, we can assume that
\[ E_{s/2}(A) > |A|^{s- k + \Lambda + \delta} =   |A|^{s - k + \Lambda(1 +  T_k^{-1})}, \]
in which case, we may iterate this argument several times to obtain the inequality
\[ E_{s/2^{r}}(A) > |A|^{s/2^{r-1}-k + \Lambda(1 +  T_k^{-1})^{r}}. \]
Since $s=2^{r+1}$, we see that the trivial bound $E_{s/2^r}(A) \leq |A|^{s/2^{r-1} - 1}$ combines with the above inequality to furnish the estimate $\Lambda(1 +  T_k^{-1})^{r} < k -1.$ Simplifying this, we get $\Lambda (s/2)^{\log (1 + T_k^{-1})} < k-1,$ that is, 
\[ \Lambda < (k-1) 2^{\log (1+ T_k^{-1})}s^{- \log (1+ T_k^{-1})}. \]
Since $T_k >1$, the number $\log (1+ T_k^{-1})$ is a positive constant lying between $0$ and $1$, thus giving us
\[ \Lambda < (2k-2) s^{-\log (1+ T_k^{-1})}, \]
and so, we finish the proof of Lemma $\ref{thrt}$.
\end{proof}

We may combine the above lemma with the trivial estimate $\eqref{shaday}$ to show that Theorem $\ref{rtp}$ holds whenever $s = 2^{r+1}$ for some $r \in \mathbb{N}$. 
\par

Now, let $s \geq 4$ and $r \geq 1$ be natural numbers such that $2^r < s < 2^{r+1}$. As before, we may now apply Lemma $\ref{thrt}$ to discern that either $E_{2^r}(A) \ll_{r,k} |A|^{2^{r+1} - k + (2k-2) 2^{-r \eta_k}}$ or $|A| \ll_{s,k} K^{O_{s,k}(1)}$. In the latter case, we can again apply $\eqref{shaday}$ to obtain the desired bound, and so, it suffices to consider the former case. Here, we can use Lemma $\ref{yoc}$ with $l= 2^r$ to infer that
\[ E_{s}(A) \leq E_{2^r}(A) |A|^{2s - 2^{r+1}} \ll_{s,k} |A|^{2s - k + (2k-2) 2^{-r\log (1+ T_k^{-1})}}. \]
 Furthermore, since $s < 2^{r+1}$, we have $(\log s - 1)\log (1+ T_k^{-1}) < r \log (1+ T_k^{-1}),$ which, in turn, implies that 
\[ 2^{-r\log (1+ T_k^{-1})} < s^{-\log (1+ T_k^{-1}) } 2^{\log (1+ T_k^{-1}) } \leq 2s^{-\log (1+ T_k^{-1}) } . \]
Inserting this in the preceding inequality finishes the proof of Theorem $\ref{rtp}$.


\section{Proof of Proposition $\ref{kp}$}

In order to prove this result, we closely follow a combination of ideas from our work on \cite[Theorem $1.1$]{Ak2020}, which itself involved a heavy utilisation of the methods in \cite{BC2011}.
\par

For the rest of this section, we will fix $s \geq 4$ to be an even natural number and $A$ to be some finite subset of $\mathbb{R}$. We begin by claiming that it suffices to prove our theorem in the case when $E_{s}(A) = |A|^{2s-\nu}$. In order to see this, note that if $E_{s}(A) > |A|^{2s-\nu}$, then we can write $E_{s}(A) = |A|^{2s - \nu'}$ for some $1 \leq \nu' <\nu$. Moreover, assuming the theorem to hold for sets $A$ and real numbers $\nu'$ satisfying $E_{s}(A) = |A|^{2s-\nu'}$, we would have that either
\[ E_{s/2}(A) > |A|^{2s - \nu' + \delta}, \] 
or there exists some subset $A'$ of $A$ such that $|A'| \gg |A|^{1 - 82\delta}$ and 
\[ |mA'-nA'| \ll_{m,n} |A|^{\nu' + 240(m+n)\delta} \]
for each $n,m \in \mathbb{N} \cup \{0\}$. Observing the fact that $\nu' < \nu$ confirms our claim.
\par


Thus, from this point onwards, we may assume that $A$ satisfies
\begin{equation} \label{asm34}
E_{s}(A) = |A|^{2s-\nu},
\end{equation}
and that
\begin{equation} \label{assum1}
E_{s/2}(A) \leq |A|^{s - \nu   + \delta}, 
\end{equation}
since if the latter inequality does not hold, we are done. We will now utilise these inequalities to construct a hypergraph $G$ on $A^s$ such that there are few distinct sums of the form $a_1 + \dots + a_s$, with $(a_1, \dots, a_s) \in G$. Thus, for each $t \in \mathbb{N}$ and for each $H \subseteq A^{t}$, we define the restricted sumset
\[  \Sigma(H) = \{ a_1 + \dots + a_t \ | \ (a_1, \dots, a_t) \in H\},\]
 and with this notation in hand, we record the following lemma. 
 
 \begin{lemma} \label{7lem1}
Let $A$ be a finite set of real numbers satisfying inequalities $\eqref{asm34}$ and $\eqref{assum1}$. Then there exists a hypergraph $G \subseteq A^s$ such that 
 \[ |G| > |A|^{s-\delta}/2 \ \text{and} \ |\Sigma(G)| \leq 4|A|^{\nu}. \]
\end{lemma}

\begin{proof}
Writing $S = \{ n \in \mathbb{R} \ | \ r_{s}(n) \geq 2^{-1}|A|^{s- \nu   } \},$ we see that 
\[ \sum_{n \notin S} r_{s}(n)^2 <  2^{-1}|A|^{s- \nu   } \sum_{n} r_{s}(n) = 2^{-1} |A|^{2s - \nu  }, \]
which can then be combined with $\eqref{sda}$ and $\eqref{asm34}$ to get
\[ \sum_{n \in S} r_{s}(n)^2 =  E_{s}(A) - \sum_{n \notin S} r_{s}(n)^2 > 2^{-1} |A|^{2s - \nu   }.\]
We can now use Lemma $\ref{yoc}$ with $\eqref{assum1}$ and the preceding inequalities to deduce that
\[ |A|^{s - \nu   + \delta} \sum_{n \in S} r_{s}(n) \geq \sum_{n \in S} r_{s}(n)^2 > 2^{-1} |A|^{2s - \nu   }, \]
and consequently, we have $\sum_{n \in S} r_{s}(n) > 2^{-1} |A|^{s- \delta}.$ Moreover, we observe that
\[  |A|^{2s - \nu   } = E_{s}(A) \geq \sum_{n \in S} r_{s}(n)^2 \geq (2^{-1}|A|^{s- \nu   })^2 |S|, \]
whence, $|S| \leq 4 |A|^{\nu   }.$ With these bounds in hand, we define our hypergraph $G \subseteq {A}^s$ as $G = \{ (a_1, \dots, a_s) \in {A}^s \ | \ \sum_{i=1}^{s} a_i \in S \}.$ This implies that
\[ |G| = \sum_{n \in S} r_{s}(n) > 2^{-1}|A|^{s- \delta} \ \text{and} \ |\Sigma(G)| = |S| \leq 4|A|^{\nu},\]
and so, we finish the proof of our lemma.
\end{proof}

We now perform some standard graph theoretic pruning in order to obtain large subsets of $A^{s/2}$ satisfying suitable combinatorial properties. We start this step of our proof by introducing some notation, and so, given any $t \in \mathbb{N}$, any $H' \subseteq {A}^{t}$, any $n\in \mathbb{R}$ and any $\vec{y} = (y_1, \dots, y_t) \in {A}^{t}$, we define 
\[ \Sigma(\vec{y}) = y_1 + \dots + y_t \ \text{and} \ r(H';n) = |\{ \vec{y} \in H' \ | \ \Sigma(\vec{y})  = n \}|. \]
As before, a simple double counting argument delivers the expression
\begin{equation} \label{stup3}
 \sum_{n \in \Sigma(H')} r(H';n) = |H'|. 
 \end{equation}
Moreover, given $u,v \in \mathbb{N}$, and $\vec{w} \in A^{u}, \vec{w}' \in A^{v}$, we write $(\vec{w},\vec{w}') = (w_1, \dots, w_{u}, w_1', \dots, w_{v}')$. Finally, given $\vec{x} \in {A}^{s/2}$ and $H \subseteq {A}^{s}$, we write 
\[R_{H}(\vec{x}) = \{ \vec{y} \in {A}^{s/2} \ | \ (\vec{x},\vec{y}) \in H  \}.\]
\par

We note that
\[ \sum_{\vec{y} \in {A}^{s/2} } \sum_{\vec{x} \in {A}^{s/2}} \mathds{1}_{R_{G}(\vec{x})}(\vec{y}) =  \sum_{\vec{x} \in {A}^{s/2}} |R_{G}(\vec{x})| = |G| > 2^{-1}|A|^{s- \delta}, \]
whence, a straightforward application of the Cauchy-Schwarz inequality gives us
\[\sum_{\vec{y} \in {A}^{s/2} }  ( \sum_{\vec{x} \in {A}^{s/2}}  \mathds{1}_{R_{G}(\vec{x})}(\vec{y}))^2  > 2^{-2} |A|^{2s - 2\delta} |A|^{-s/2}. \]
Simplifying the above, we get
 \[ \sum_{\vec{x},\vec{x}' \in {A}^{s/2}} |R_{G}(\vec{x}) \cap R_{G}(\vec{x}')|  = \sum_{\vec{y} \in {A}^{s/2}} \sum_{\vec{x},\vec{x}' \in {A}^{s/2}}  \mathds{1}_{R_{G}(\vec{x})}(\vec{y})  \mathds{1}_{R_{G}(\vec{x}')}(\vec{y}) > 2^{-2} |A|^{3s/2 - 2\delta}, \]
whereupon, we note that there must exist some $\vec{x}$ in ${A}^{s/2}$ such that
\[ \sum_{\vec{x}' \in {A}^{s/2}} |R_{G}(\vec{x}) \cap R_{G}(\vec{x}')| > 2^{-2} |A|^{s - 2\delta}. \]
This leads us to define the hypergraph $G_1 = \{ (\vec{y},\vec{z}) \in G \ | \ \vec{y} \in {A}^{s/2} \ \text{and} \ \vec{z} \in R_{G}(x) \}$, in which case, the preceding inequality implies that
\[ |G_1| = \sum_{\vec{y} \in {A}^{s/2}} |R_{G}(\vec{y}) \cap R_{G}(\vec{x})| > 2^{-2} |A|^{s - 2\delta}. \]
\par

Writing 
\begin{equation} \label{ydef}
Y = \{ \vec{y} \in {A}^{s/2} \ | \ |R_{G}(\vec{y}) \cap R_{G}(\vec{x})| \geq 2^{-3} |A|^{s/2 - 2\delta} \}, 
\end{equation}
we observe that
\[ \sum_{\vec{y} \notin Y} |R_{G}(\vec{y}) \cap R_{G}(\vec{x})| < 2^{-3} |A|^{s/2 - 2\delta} \sum_{\vec{y} \notin Y} 1 \leq 2^{-3} |A|^{s - 2\delta}, \]
and as a result, we discern that
\[ \sum_{\vec{y} \in Y} |R_{G}(\vec{y}) \cap R_{G}(\vec{x})| = \sum_{\vec{y} \in {A}^{s/2}} |R_{G}(\vec{y}) \cap R_{G}(\vec{x})|  -  \sum_{\vec{y} \notin Y} |R_{G}(\vec{y}) \cap R_{G}(\vec{x})| > 2^{-3}|A|^{s - 2\delta}. \]
This implies that $|Y| |A|^{s/2} > 2^{-3} |A|^{s - 2\delta},$ whenceforth, we have $|Y| > 2^{-3} |A|^{s/2 - 2 \delta}$. Furthermore, since
\[ \sum_{\vec{z} \in R_{G}(\vec{x})} |  \{ \vec{y} \in Y \ | \ (\vec{y},\vec{z}) \in G_1  \}| = \sum_{\vec{y} \in Y} |R_{G}(\vec{y}) \cap R_{G}(\vec{x})| > 2^{-3}|A|^{s - 2\delta}, \]
there exists some $\vec{z} \in R_{G}(x)$ such that
\begin{equation} \label{lowerbdy1}
 |\{ \vec{y} \in Y \ | \ (\vec{y},\vec{z}) \in G_1\}| > 2^{-3} |A|^{s/2 - 2 \delta}.
 \end{equation}
Fixing such a $\vec{z}$, we write $Y_1 = \{ \vec{y} \in Y \ | (\vec{y},\vec{z}) \in G_1\}$. 
\par

Next, we consider the set $S_1 \subseteq \Sigma(Y_1)$, where
\[ S_1 = \bigg\{ \ n \in \Sigma(Y_1) \ \bigg| \ r(Y_1; n) >  \frac{|Y_1|}{2|\Sigma(Y_1)| } \ \bigg\}. \]
Upon defining $Y_2 = \{ \vec{y} \in Y_1 \ | \ \Sigma(\vec{y}) \in S_1 \}$, we see that
\[ |Y_1 \setminus Y_2| = \sum_{n \in \Sigma(Y_1 \setminus Y_2) } r(Y_1\setminus Y_2; n)  \ \leq \ 2^{-1} | \Sigma(Y_1)|^{-1}  |Y_1| |\Sigma(Y_1 \setminus Y_2)| \ \leq \ 2^{-1} |Y_1|,\]
which, when amalgamated with $\eqref{lowerbdy1}$, delivers the bound
\begin{equation} \label{lwbdy2}
 |Y_2| = |Y_1| - |Y_1 \setminus Y_2| \geq 2^{-1} |Y_1| > 2^{-4} |A|^{s/2 - 2 \delta}. 
 \end{equation}
We recall that $Y_2 \subseteq Y_1 \subseteq Y$, and so, $\eqref{ydef}$ implies that for each $\vec{y} \in Y_2$, we have
\begin{equation} \label{rdlbd}
|R_{G}(\vec{y}) \cap R_{G}(\vec{x})| \geq 2^{-3} |A|^{s/2 - 2\delta}.
\end{equation}
Lastly, we claim that
\begin{equation}  \label{hndsgt}
 |\Sigma(Y_2)| \leq | \Sigma(Y_1)| \leq |\Sigma(G_1)|, \ \text{and} \ |\Sigma(R_{G}(\vec{y}) \cap R_{G}(\vec{x}))| \leq |\Sigma(G_1) |
 \end{equation}
for each $\vec{y} \in Y_2$. In order to see this, note that for each $\vec{y} \in Y_1$, we have $(\vec{y}, \vec{z}) \in G_1$, and so, $\Sigma(Y_1) + \Sigma(\vec{z}) \subseteq \Sigma(G_1)$. This combines with the fact that $Y_2 \subseteq Y_1$ to deliver the first inequality stated in $\eqref{hndsgt}$. Similarly, we may deduce the second inequality in $\eqref{hndsgt}$ by noting that for each $\vec{y}' \in R_{G}(\vec{x})$, we have $(\vec{x}, \vec{y}') \in G_1$. 
\par


With the sets $Y_2$ and $R_{G}(\vec{x})$ in hand, we will now use their suitable combinatorial properties to extract some additive structure between $\Sigma(Y_2)$ and $\Sigma(R_{G}(\vec{x}))$. This allows for a combined application of the Balog--Szemer\'edi--Gowers theorem and the Pl{\"u}nnecke--Ruzsa theorem, which is what we will proceed with in the forthcoming lemma. For ease of exposition, we denote $ M = 4|A|^{\nu}$ and $\alpha = 2^{-37} |A|^{-20 \delta}$.

\begin{lemma} \label{7lem2}
There exists a set $U' \subseteq \Sigma(Y_2)$ such that for every $m, n \in \mathbb{N} \cup \{0\}$, we have
\begin{equation} \label{mdn}
 |U'| \geq 2^{-20} \alpha^4 M, \ \text{and} \ |mU'- nU'| \leq (2^{62} \alpha^{-12})^{m+n} |U'|.
 \end{equation}
\end{lemma}

\begin{proof}
We begin our proof by applying the Cauchy-Schwarz inequality on $\eqref{stup3}$ to infer that
\[ |\Sigma(H)| \sum_{n \in \Sigma(H)} r(H; n)^2 \geq |H|^2, \]
for any $H \subseteq A^{s/2}$. Moreover, since $\sum_{n \in \Sigma(H)} r(H; n)^2  \leq E_{s/2}(A)$, we can combine the preceding inequality and $\eqref{assum1}$ to note that
\[ |\Sigma(H)| \geq |H|^2 E_{s/2}(A)^{-1} \geq |H|^2 |A|^{-s + \nu    -\delta}. \]
Substituting $H = Y_2$ in the above and combining this with $\eqref{lwbdy2}$, we deduce that
\[ |\Sigma(Y_2)|  > 2^{-8} |A|^{s - 4 \delta}|A|^{-s + \nu    -\delta} = 2^{-8} |A|^{\nu   - 5\delta}. \]
Similarly, noting $\eqref{rdlbd}$, we infer that for each $\vec{y} \in Y_2$, we have
\[ |\Sigma(R_{G}(\vec{y}) \cap R_{G}(\vec{x})| \geq 2^{-6} |A|^{s - 4\delta}|A|^{-s + \nu    -\delta} = 2^{-6} |A|^{\nu    - 5 \delta}. \]
Thus, upon writing $U = \Sigma(Y_2)$ and $V = \Sigma(R_{G}(\vec{x}))$ and defining
\[ r(U,V; n) = |\{(u,v) \in U \times V \ | \ u + v = n \}| \]
for each $n \in \mathbb{R}$, we see that the preceding discussion implies that
\begin{align*}
 \sum_{n \in \Sigma(G_1)} r(U,V; n) & =\sum_{n \in \Sigma(G_1)}| \{   (a,b) \in \Sigma(Y_2) \times \Sigma(R_{G}(\vec{x})) \ | \ n = a+ b  \}| \\
 & \geq |\Sigma(Y_2)| \min_{\vec{y} \in Y_2} |\Sigma(R_{G}(\vec{y}) \cap R_{G}(\vec{x}))| \geq2^{-14} |A|^{2\nu -  10 \delta}.
 \end{align*}
Applying the Cauchy-Schwarz inequality on the left hand side above, we get
\[ | \Sigma(G_1)|  \sum_{n \in \Sigma(G_1)} r(U,V; n)^2 \geq 2^{-28} |A|^{4\nu  - 20 \delta}. \]
Since $G_1 \subseteq G$, we have $|\Sigma(G_1)| \leq 4|A|^{\nu   }$, and as a result, we deduce that
\[ \sum_{n \in \Sigma(G_1)} r(U,V; n)^2 \geq 2^{-28} |A|^{4\nu  - 20 \delta} |\Sigma(G_1)|^{-1} \geq 2^{-30} |A|^{3\nu - 20 \delta }. \]
\par

We may rewrite the above inequality as
\[ \sum_{n \in \Sigma(G_1)} r(U,V; n)^2 \geq 2^{-30} |A|^{- 20 \delta } |A|^{3\nu} = 2\alpha M^3, \]
whereupon, we define $\mathcal{S} = \{ n \in U + V \ | \ r(U,V; n) \geq \alpha M \}$. If $|\mathcal{S}| > M$, we let $\mathcal{S}'$ be a subset of $\mathcal{S}$ such that $|\mathcal{S}'| = M$, and in this case, we trivially have
\[ \sum_{n \in \mathcal{S}'} r(U,V; n) \geq  \alpha M |S'| = \alpha M^2 . \]
On the other hand, if $|\mathcal{S}| \leq M$, we let $\mathcal{S}' = \mathcal{S}$, and note that
\begin{align*} 
 \sum_{n \in \mathcal{S}'} r(U,V; n)^2 
 & = \sum_{n \in U+V} r(U,V; n)^2 - \sum_{n \notin \mathcal{S}'} r(U,V; n)^2  > 2\alpha M^3 - \alpha M \sum_{n \notin \mathcal{S}'} r(U,V; n) \nonumber \\ 
 & \geq 2 \alpha M^3  - \alpha M |U| |V|  \geq 2 \alpha M^3 - \alpha M^3 = \alpha M^3. 
 \end{align*}
Moreover recalling $\eqref{hndsgt}$, we see that $|U|, |V| \leq |\Sigma(G)| \leq 4|A|^{\nu} = M$, and so, we get $r(U,V; n) \leq |U| \leq M$. Combining this with the preceding discussion then gives us
\[\sum_{n \in \mathcal{S}'} r(U,V; n) \geq \alpha M^2 . \]
\par

In either case, we define $\mathcal{G} \subseteq U \times V$ to be $\mathcal{G} = \{ (u,v) \in U \times V \ | \ u + v \in \mathcal{S}' \},$ and note that $|\mathcal{G}|  = \sum_{n \in \mathcal{S}'} r(U,V; n) \geq \alpha M^2.$ 
We may now use Lemma $\ref{balbsg}$ to deduce the existence of $U' \subseteq U$ such that 
\[ |U'| \geq 2^{-16} 3 \alpha^3 M (\log (32/\alpha))^{-1} \geq 2^{-20} \alpha^{4} M \]
and
\[ |U' + U'| \leq 2^{38} 3^{-1} \log(32/\alpha) \alpha^{-7} M \leq 2^{42} \alpha^{-8} M.\]
In particular, these imply that
\[ |U'+U'| \leq 2^{42} \alpha^{-8} 2^{20} \alpha^{-4} |U'| = 2^{62} \alpha^{-12} |U'|, \]
which further combines with Lemma $\ref{pr21}$ to give us
\[  |mU'- nU'| \leq (2^{62} \alpha^{-12})^{m+n} |U'| \]
for every $m,n \in \mathbb{N} \cup \{0\}$, and so, we are done.
\end{proof}

Choosing the set $U' \subseteq U = \Sigma(Y_2)$ as in the conclusion of Lemma $\ref{7lem2}$, we may define $Y_3 = \{ \vec{y} \in Y_1 \ | \ \Sigma(\vec{y}) \in U' \},$ and note that
\[  |Y_3|   =   \sum_{n \in U'}  r(Y_1; n) \geq |U'|   \frac{|Y_1|}{2| \Sigma(Y_1)| }  \geq  2^{-21} \alpha^{4} M \frac{|Y_1|}{ |\Sigma(G)| }  \geq 2^{-21} \alpha^{4} |Y_1|,\]
where the last two inequalities follow from combining inequalities $\eqref{hndsgt}$ and $\eqref{mdn}$ along with the fact that $|\Sigma(G_1)| \leq |\Sigma(G)| \leq M$. Putting this together with the lower bound $\eqref{lowerbdy1}$ for $|Y_1|$, we find that
\[ |Y_3| \geq  2^{-24} \alpha^{4} |A|^{s/2 - 2 \delta} .\]
 Moreover, since $Y_3 \subseteq {A}^{s/2}$, there exists some $\vec{w} \in {A}^{s/2-1}$ and a subset ${A}' \subseteq {A}$ such that 
\[ |{A}' | \geq |Y_3| |A|^{-s/2 + 1}  \geq 2^{-24} \alpha^{4} |A|^{1 - 2 \delta}, \]
 and $(\vec{w},a) \in Y_3 \ \text{for each} \ a \in {A'}.$ In particular, this implies that for each $a \in {A'}$, the element $\Sigma(\vec{w}) + a \in \Sigma(Y_3)$, whereupon, for each $m, n \in \mathbb{N} \cup \{0\}$, we have
\[ m {A'} - n {A'} + (m-n)\Sigma(\vec{w}) \subseteq m \Sigma(Y_3) - n \Sigma(Y_3) \subseteq m U' - nU'. \]
Combining this with $\eqref{mdn}$, we conclude that
\[ |m {A'} - n {A'}| \leq (2^{62} \alpha^{-12})^{m+n} |U'| \leq (2^{62} \alpha^{-12})^{m+n} M. \]
Substituting $M = 4|A|^{\nu}$ and $\alpha = 2^{-37} |A|^{-20 \delta}$ in the above, we see that
\[  |{A}' | \geq 2^{-24} \alpha^{4} |A|^{1 - 2 \delta} = 2^{-172} |A|^{1-82 \delta}, \ \text{and} \ |m {A'} - n {A'}| \leq 2^{506(m+n)+2}  |A|^{\nu+ 240 (m+n) \delta}, \]
and consquently, we finish our proof of Proposition $\ref{kp}$.


\bibliographystyle{amsbracket}
\providecommand{\bysame}{\leavevmode\hbox to3em{\hrulefill}\thinspace}

\end{document}